\documentclass[12pt]{amsart}
\usepackage[utf8]{inputenc}
\usepackage[T1]{fontenc}
\usepackage{amsmath}
\usepackage{amsfonts}
\usepackage{amssymb,amsthm}
\usepackage{graphicx}
\usepackage{mathrsfs}
\usepackage{hyperref}
\usepackage{color}
\usepackage[dvipsnames]{xcolor}
\usepackage{comment}
\usepackage{appendix}
\usepackage{enumerate}
\usepackage{enumitem}
\usepackage{cancel}
\usepackage{float} 
\usepackage{bm}

\newtheorem{thm}{Theorem}[section]

\newtheorem{pro}[thm]{Proposition}
\newtheorem{cor}[thm]{Corollary}
\newtheorem{lem}[thm]{Lemma}
\newtheorem{rem}[thm]{Remark}

\newcommand{\R}{\mathbb{R}}
\newcommand{\N}{\mathbb{N}}

\usepackage{chngcntr}

\counterwithin*{equation}{section}
\def\qq#1{\qquad \mbox{#1}\quad}
\def\q#1{\quad \mbox{#1}\ }
\newcommand{\al}{\alpha}
\newcommand{\be}{\beta}

\newcommand{\e}{\varepsilon }

\newcommand{\g}{\gamma }

\newcommand{\La}{\Lambda }

\newcommand{\na}{\nabla }

\newcommand{\Om}{\Omega }
\newcommand{\Omb}{\overline{\Om}}
\newcommand{\p}{\partial }
\newcommand{\s}{\sigma}
\newcommand{\te}{\theta}
\newcommand{\Te}{\Theta}

\title[$L^\infty$ estimates in the presence of gradient terms]{$L^\infty$ estimates for solutions to elliptic equations in the presence of gradient terms}

\author[E.~Antonio]{Edgar Antonio}
\address[E.~Antonio]{Universidad Autónoma de Guerrero, Guerrero, Mexico \\ Universidad Complutense de Madrid, Madrid, Spain}
\email{eaam020713@gmail.com | eantonio@ucm.es}
\author[R.~Pardo]{Rosa Pardo}
\address[R.~Pardo]{Universidad Complutense de Madrid, Madrid, Spain}
\email{rpardo@ucm.es}

\thanks{The second author is supported by grants  PID2022-137074NB-I00,  MICINN,  Spain, and by UCM, Spain,  Grupo 920894.}

\date{}

\begin{document}
%----------------------------------------------------------------
%----------------------------------------------------------------
%--------------------- A B S T R A C T --------------------------
%----------------------------------------------------------------
%----------------------------------------------------------------
\begin{abstract}
We consider an elliptic problem with slightly subcritical nonlinearities at the interior and on the boundary; the nonlinearity at the interior is also depending on a gradient term.
For any $u$ weak solution, we provide explicit $L^\infty$ {\it a priori} estimates depending only on both nonlinearities, on the $H^1(\Om)$ norm of $u,$  and on the domain $\Omega$. To obtain our results, we combine De Giorgi-Nash-Moser iteration procedure,   elliptic regularity when  the gradient term appears,  Lebesgue interpolation and the Gagliardo-Nirenberg interpolation inequalities. 
\end{abstract}
\maketitle
{\bf MSC 2020: }{\it Primary 35B45; %A priori estimates in context of PDEs
Secondary 
35J66,  %Nonlinear boundary value problems for nonlinear elliptic equations
35B33,  %Critical exponents in context of PDEs
%35J75,  % Singular elliptic equations 
35J25. %BVP for second-order elliptic equations, 
%35J60,   	%Nonlinear elliptic equations
%35J61,   	%Semilinear elliptic equations
}


%----------------------------------------------------------------
%----------------------------------------------------------------
%----------------- I N T R O D U C T I O N ----------------------
%----------------------------------------------------------------
%----------------------------------------------------------------

\section{Introduction}
The main goal of this paper is to provide an explicit   $L^\infty(\Om)$ estimate of the weak solutions to the problem
\begin{equation}\label{pde}
\left \{
\begin{aligned}
-\Delta u +u=& f(x,u, \nabla u ), \; \; \; x\in \Om , \\
\frac{\p u}{\p \nu }=& f_{B}(x,u), \;x\in \p \Om , 
\end{aligned}
\right .
\end{equation}
in a certain open, bounded, connected  domain $\Om \subset \mathbb{R}^N$, $N> 2$, with $C^{0,1}$ (Lipschitz) boundary, where the nonlinearities, at the interior $f:\Om\times \mathbb{R} \times \mathbb{R}^N \rightarrow \mathbb{R}$ and on the boundary $f_{B}:\p \Om \times \mathbb{R} \rightarrow \mathbb{R}$ are both {\it slightly subcritical}  Carathéodory functions
(see \ref{f2} and \ref{fB2} respectively for a precise definition of slightly subcritical function). \\

\par We say that a weak solution $u\in H^1(\Om)$ of \eqref{pde} has an $L^\infty(\Om)$ {\it a priori estimate} if  $u \in L^\infty(\Om)$ and there exists some function $h:\R\to \R$ depending only on $f$, and $f_{B}$,  satisfying $h(s)\to\infty$ as $s\to\infty$, and such that 
\begin{equation*}%\label{estim:u}
h\big(\|u\|_{L^\infty(\Om)} \big)\leq M,\qq{where}M=M\big(\|u\|_{H^1(\Om)},\Om \big) .
\end{equation*}
Moser iteration techniques imply $L^{\infty}(\Om)\cap L^{\infty}(\p\Om)$ regularity of weak solutions to \eqref{pde} when $f, f_B$ are both {\it slightly subcritical} or even critical, see \cite[Theorem~3.1]{Marino_Winkert_NonlAnal_2019}, \cite[Theorem~4.1]{Ho_Winkert_Zhang_2022}, and  \cite[Theorem~4.3]{Ho_Winkert_2023}, where more general quasilinear problems, including  as a particular case \eqref{pde}, are treated. 
By elliptic regularity, it turns out that if $u$ is a weak solution to \eqref{pde}, then $u \in  W^{1,m}(\Om )\cap  C^{\nu }(\Omb)$ as well (see Theorem \ref{th:reg:nol}), with $m>N>2$ and $\nu=1-\frac{N}{m}$.

\bigskip

Our goal here is to  explicitly define the function $h$  in terms of $f$, and $f_{B}$, and also the bound $M$  in terms of powers of the $H^1(\Om)$ norm of $u$, see Theorem \ref{th:RegEst}. 
Our estimates are valid either for  positive or changing-sign solutions. Consequently, any sequence of solutions to \eqref{pde}, {\it uniformly bounded} in the $H^1(\Om)$ norm, is also {\it uniformly bounded} in the $L^\infty (\Om)$ norm.
\bigskip

For similar explicit $L^{\infty}(\Om)$ estimates in the semilinear and quasilinear cases satisfying homogeneous Dirichlet boundary conditions with  nonlinearities exclusively in the interior, see \cite{Pardo_JFPTA_2023}  and \cite{Pardo_RACSAM_2024}. See also \cite{Chhetri_Mavinga_Pardo_PAMS, CMP_NonlA_2026} for elliptic equations and systems,  with a subcritical nonlinearity exclusively on the boundary. For the combination of slightly subcritical nonlinearities at the interior and on the boundary, see \cite{Antonio_Arciga_Pardo_Sanchez}. 
None of the previous explicit estimate results  include gradient terms in the nonlinearities, which is the core of our present work.
\bigskip

One of the main difficulties in addressing these problems depends largely on obtaining adequate regularity estimates for the gradient term;  see Theorem \ref{th:W1m:unif} and Theorem \ref{th:reg:nol} for that purpose. \\

\par A uniform $H^1(\Om)$ {\it a priori} bound result for weak solutions to \eqref{pde} will complement our result, providing a uniform $L^\infty(\Om)$ {\it a priori} bound. Theorem \ref{th:RegEst} implies in particular that sequences of weak solutions to \eqref{pde} uniformly  bounded in  their $H^1(\Om)$-norms, are uniformly bounded in their $L^\infty(\Om)$-norm.
In fact, a question naturally arises: does the equation  \eqref{pde} have a uniform $L^\infty(\Om)$ {\it a priori} bound for any positive weak solution? 
To the best of our understanding, it remains as an open problem.  
\bigskip

The main novelties of this paper are the following:
\bigskip

(i) We provide explicit $L^\infty$  {\it a priori} estimates for solutions to nonlinear elliptic equations in terms of their $H^1-$norms, see Theorem \ref{th:RegEst}. To obtain it, we need to analyze carefully a related problem \eqref{eq:W1m:fx:Unif}, obtaining {\it a priori} estimates for solutions to that auxiliar problem, see Theorem \ref{th:reg:2} and Theorem \ref{th:W1m:unif}. 

\medskip

(ii) We consider nonlinearities depending on the gradient.

\medskip

(iii) We focus on slightly subcritical functions, that is, at the interior $f=f(x,s,\xi)$ satisfies \ref{f1}-\ref{f2}; and on the boundary $f_{B}=f_B(x,s)$ satisfies \ref{fB1}-\ref{fB2}.

\medskip

(iv) As a Corollary, we provide explicit $L^\infty$  {\it a priori} estimates for solutions to nonlinear elliptic equations in terms of their critical Lebesgue $L^{2^*}(\Om)$ and trace $L^{2_*}(\p\Om)$ norms, see Corollary \ref{cor1}. 

\bigskip

This paper is outlined as follows.
We devote Section \ref{sec:prel} to some preliminaries, including the precise hypothesis on the nonlinearities, to state our main result, Theorem \ref{th:RegEst}. Section \ref{sec:reg:estim} contains a careful analysis of some estimates for weak solutions to an auxiliary  problem \eqref{eq:W1m:fx:Unif} with $g=g(x,\xi)$, and nonhomogeneous Neumann boundary conditions, see Theorem \ref{th:W1m:unif}. Section \ref{sec:expl:estim}, is devoted to  prove  Theorem \ref{th:RegEst} and Corollary \ref{cor1}. 
In Section \ref{sec:expl} we explicitly calculate  the exponent $\be_0$ in  Theorem \ref{th:RegEst}.
Finally, in  the Appendices \ref{sec:appA} and \ref{sec:appB} we include some Moser type results, adding some appropriate elliptic regularity estimates, see Theorem \ref{th:A} and Theorem \ref{th:reg:nol}
From now on, throughout this paper $C$ denotes several constants independent of $u$.

%----------------------------------------------------------------
%----------------------------------------------------------------
%---------------- P R E L I M I N A R I E S ---------------------
%----------------------------------------------------------------
%----------------------------------------------------------------

\section{Preliminaries and Main Result}\label{sec:prel}
For $p>1$, we define the trace operator as follows
\begin{equation*} 
\Gamma :W^{1,p}(\Om)\rightarrow L^{p}(\p \Om ),
\end{equation*}

\begin{enumerate}[leftmargin=.2cm]
\item $\Gamma u= u|_{\p \Om }$ \,  if \, $u\in W^{1,p}(\Om )\cap C(\Omb)$.
\bigskip  
\item $\|\Gamma u\|_{L^{p}(\p \Om )}\leq C\|u\|_{W^{1,p}(\Om )}$.
\end{enumerate}
\noindent Since the surjectivity and the continuity of the trace operator, we get 
\begin{equation*} 
\Gamma :W^{1,p}(\Om)\rightarrow W^{1-\frac{1}{p},p}(\p \Om ) \hookrightarrow L^{q}(\p \Om ), \quad \text{for} \quad  1\le q\le \frac{(N-1)p}{N-p},
\end{equation*}
and 
\begin{equation*}
\|\Gamma u \|_{L^{q}(\p \Om )}\le C\|u\| _{W^{1,p}(\Om)}, \quad \text{for some} \quad C>0.
\end{equation*}
This operator is continuous for $1\le q\le \frac{(N-1)p}{N-p}$, and compact for  $1\le q< \frac{(N-1)p}{N-p}$ (see  \cite[Theorem 6.4.1]{Kufner_Fucik} and \cite[Lemma 9.9]{Brezis}). 
\medskip

Throughout this paper, we use the Sobolev embedding
$
H^{1}(\Om )\hookrightarrow L^{2^*}(\Om ),
$ 
and  the continuity of the trace operator 
$ 
H^{1}(\Om )\hookrightarrow L^{2_*}(\p \Om ) ,
$
where $2^*:=\frac{2N}{N-2}$ is the critical Sobolev exponent and $2_*:=\frac{2(N-1)}{N-2}=\frac{(N-1)}{N}\,2^*$ is the critical exponent in the sense of the trace.\\
For $1< p,\ p_B\le \infty$,  the critical Hardy-Sobolev exponent are given by
\begin{equation}\label{Def:Np:NpB}
2^*_{N/p}:=\frac{2^*}{p'}=2^*\Big( 1-\frac{1}{p}\Big) \quad \text{and}\quad 2_{*,N/ p_{B}}:=\frac{2_*}{p_{B}'}=2_*\Big( 1-\frac{1}{p_B}\Big) ,
\end{equation}
where $p'$ is the conjugate exponent of $p$, that is $\frac{1}{p}+\frac{1}{p'}=1$. \\

Given $f: \Om \times \mathbb{R} \times \mathbb{R}^N \to \mathbb{R} $, we assume   the following hypothesis on the nonlinearity at the interior
\begin{enumerate}[label=\textbf{(f\arabic*)}, leftmargin=.2cm]
\item
\label{f1}
$f$ is a {\it Carath\'eo\-dory} function:
\begin{enumerate}[leftmargin=.2cm]
\item $f(\cdot, t , \xi) $ is measurable for each  $(t,\xi ) \in \mathbb{R} \times \mathbb{R}^N$;
\item $f(x, \cdot , \cdot )$ is continuous for each $x \in \Om$.
\end{enumerate}
\item 
\label{f2} $f$ is {\it slightly subcritical (at infinity)}, that is
\begin{equation*}
|f(x,s,\xi)|  \le  a_{1}|\xi|^{l} +  \widehat{f}(x,s) ,
\end{equation*}
where $0< l<\tfrac{N+2}{N}$, $a_1\in \mathbb{R}$, 
\begin{align}\label{f:hat}
\widehat{f}(x,s):=|a_{2}(x)| &\widetilde{f}\big(|s|\big),
\end{align}
$a_2\in L^{r}(\Om)$ for $r>\frac{N}{2}$, $\widetilde{f}:[0,+\infty )\rightarrow [0,+\infty )$
is continuous, non-decreasing, $\widetilde{f}(s) >0$ for $s>0$, and  such that
\begin{equation}\label{E0.3}
\lim _{ s\rightarrow +\infty }\frac{\widetilde{f}(s)}{ s^{\, 2^*_{N/r}-1}}=0.
\end{equation}
\end{enumerate}
\bigskip

Likewise, given $f_{B}:\p \Om \times \mathbb{R} \rightarrow \mathbb{R}$, we  assume the following hypothesis for the nonlinearity on the boundary
\begin{enumerate}[label=\textbf{(f$_B$\arabic*)},leftmargin=.2cm]
\item
\label{fB1}
$f_{B}$ is a {\it Carath\'eo\-dory} function:
\begin{enumerate}[leftmargin=.2cm]
\item $f_{B}(\cdot ,s)$ is measurable for each  $s\in \mathbb{R}$;
\item $f_{B}(x, \cdot )$ is continuous for each $x \in \p \Om$.
\end{enumerate}
\item
\label{fB2}
$f_{B}$ is {\it slightly subcritical (at infinity)}, that is:
\begin{equation*}
|f_{B}(x,s)| \leq |a_{B}(x)|\,  \widetilde{f}_{B}\big(|s|\big),
\end{equation*}
with $a_B(x)\in L^{r_B}(\p \Om )$ for $ r_{B}>N-1 $,  and $\widetilde{f}_B:[0,+\infty )\rightarrow [0,+\infty )$ is 
continuous, non-decreasing, $\widetilde{f}_B(s) >0$ for $s>0$, and such that
\begin{equation}\label{E0.5}
\lim _{ s\rightarrow +\infty }\frac{\widetilde{f}_B(s)}{ s^{\, 2_{*,N/ r_{B}}-1}}=0.
\end{equation}
\end{enumerate}
Concerning both nonlinearities $f$ and $f_B$ we also assume the following:
\begin{enumerate}[label=\textbf{(f.f$_B$)},leftmargin=.2cm]
\item
\label{ffB}
Either $\widetilde{f}(s) \to\infty$ as $s\to\infty$, or $\widetilde{f}_B(s) \to\infty$ as $s\to\infty$.
\end{enumerate}

\bigskip

\noindent We say that $u$ is a weak solution to \eqref{pde} if $u\in H^{1}(\Om ),$ $f(\cdot, u, \na u)\in L^{(2^{*})'}(\Om )$, $f_{B}(\cdot, u)\in L^{(2_{*})'}(\p \Om )$, and  for all $\psi\in H^{1}(\Om ),$
\begin{equation*}%\label{P4}
\int _{\Om }\left(\na u \na \psi + u\psi \right) dx = \int _{\Om }f(x,u,\na u)\psi\,dx+ \int _{\p \Om }f_{B}(x,u)\psi\, dS \,, 
\end{equation*}
where $ (2^*)':=\frac{2N}{N+2}$ and $(2_*)':=\frac{2(N-1)}{N} $ 
are the conjugate exponents of $2^*$ and $2_*$ respectively.

\begin{rem}%\label{remark-1}  
(i) Let $u\in H^{1}(\Om )$.  Obviously, for $0<l<\tfrac{N+2}{N},$ $|\na u|^l\in L^{2/l}(\Om)$ with $2/l>(2^*)'.$ Moreover, by Sobolev embeddings, for $f$ and $f_B$ slightly subcritical (satisfying {\rm\ref{f2}} and {\rm\ref{fB2}} respectively), we have
\begin{align*}
\widetilde{f}(|u|)\in L^{\frac{2^*}{2^*_{N/r}-1}}(\Om ), \quad\text{where}& \quad \frac{2^*_{N/r}-1}{2^*}=\frac{1}{2}+\frac{1}{N}-\frac{1}{r},\\
\widetilde{f}_{B}(|u|)\in L^{\frac{2_*}{2_{*,N/ r_{B}}-1}}(\p \Om ), \quad\text{where}& \quad \frac{2_{*,N/ r_{B}}-1}{2_*}=\frac{N}{2(N-1)}-\frac{1}{r_B}.
\end{align*}
Hence, 
\begin{equation*}
f(\cdot ,u, \na u)\in L^{(2^*)'}(\Om ) \quad \text{and} \quad  f_{B}(\cdot ,u)\in L^{(2_*)'}(\p \Om ).
\end{equation*}
(ii) We can always choose $\widetilde{f}$ and $\widetilde{f}_{B}$ such that $\widetilde{f}(s)>0$ and $\widetilde{f}_{B}(s)>0$ for $s>0$, and redefining $\widetilde{f}(s)$ and $\widetilde{f}_{B}(s)$, as $\max_{[0,s]}\widetilde{f}$ and $\max_{[0,s]}\widetilde{f}_{B}$, resp.,  $\widetilde{f}$ and $\widetilde{f}_{B}$ are non decreasing positive functions  for $s>0$.
\end{rem}

\medskip

Next, we define a new functions, $h$, which is essential for the statement of our main result. 
Given the auxiliary functions
\begin{equation}\label{def:h:hB}
\widetilde{h}(s):=\frac{s^{\, 2^*_{N/r}-1}}{\widetilde{f}(s)} \quad \text{and} \quad \widetilde{h}_{B}(s):=\frac{s^{\, 2_{*,N/ r_{B}}-1}}{\widetilde{f}_{B}(s) } ,\quad \text{for} \ s>0,
\end{equation}
let $h$  be the minimum of $\widetilde{h}$ and a certain power of $\widetilde{h}_B$, specifically
\begin{equation}\label{def:h}
h(s):=\min \left\{ \widetilde{h}(s),\, \widetilde{h}_{B}^{\,\frac{2^*_{N/r}-1}{2_{*,N/r_B}-1}}(s)\right\}\rightarrow \infty  \q{as}s\rightarrow \infty . 
\end{equation}
using \eqref{E0.3} and \eqref{E0.5}.

Our main result is the following one:
\begin{thm}[\texorpdfstring{$L^\infty (\Om )$}{}  a priori estimates] \label{th:RegEst}
Let $f:\Om\times \mathbb{R} \times \mathbb{R}^N \rightarrow \mathbb{R}$ and $f_{B}=\p \Om \times \mathbb{R}\rightarrow \mathbb{R}$ be Carath\'eodory functions, satisfying {\rm\ref{f1}}-{\rm\ref{f2}} and {\rm\ref{fB1}}-{\rm\ref{fB2}} respectively, and assume also {\rm\ref{ffB}}.
Let $u$ be a weak solution to \eqref{pde}, then,
there exist an exponent $\be_0=\be_0(N,l, r, r_B)>0$, such that
for all $\e >0$,  there exists a constant $C_\e>0$ 
\begin{equation}\label{estim}
h(\|u\|_{L^{\infty }(\Om ) })
\le C_\e\Big(1+\|u\|_{H^1(\Omega)}^{\be_0+\e} \Big),
\end{equation}
where $C_{\e }=C_\e(N,l,r, r_B,|\Om |,|\p \Om |, a_M )>0$ is independent of $u$.
\end{thm}
In the Corollary \ref{coro}  we give an explicit definition of $\be_0$, see also table \eqref{Table1} in  Appendix \ref{ApC} to  summarize the result.

An intermediate estimate in terms of powers of the critical Sobolev and trace norms is obtained  during the proof of Theorem \ref{th:RegEst}. Specifically we have the following Corollary.

  \begin{cor}\label{cor1}
Assume that all the hypothesis of Theorem \ref{th:RegEst} are satisfied.
Let $u$ be a weak solution to \eqref{pde}, then there exist three constants,   $\be_i>0$, $i=1,2,3$, depending only on $N,l, r, r_B$, and independent of $u$,   satisfying that for all $\e>0$, there exists a constant $C=C_\e$ such that  
\begin{align*} 
 h\big(\| u\|_{L^\infty (\Om ) }\big)
&\leq C_\e\Big(1+ \|u\|_{L^{2^*}(\Omega)}^{\be_1}
+ \|u\|_{L^{2_*}(\p\Omega)}^{\be_2}  \Big) \|u\|_{L^{2^*}(\Om)}^{\be_3},
\end{align*}
where $C_\e=C_\e(N, r, r_B,|\Om |,|\p \Om |,a_1, a_M )>0$  is a constant independent of $u$, and $a_M$ is  given by
\begin{equation}\label{Def:aM}
a_M:=\max \{1, \|a_2\|_{L^{r}(\Om)} , \|a_B\|_{L^{r_B}(\p\Om)}\}.
\end{equation}
\end{cor}

\section{Some elliptic regularity estimates}
\label{sec:reg:estim}
In that section, we obtain some elliptic estimates 
for the auxiliary nonlinear problem 
\begin{equation}\label{eq:W1m:fx:Unif}
\left \{
\begin{aligned}
-\Delta u +u=& g(x,\na u), \; \; \; x\in \Om , \\
\frac{\p u}{\p \nu }=& \overline{g}_{B}(x), \;x\in \p \Om ,
\end{aligned}
\right.
\end{equation}
where
\begin{align}\label{g}\tag{$g.g_B$}
&|g(x,\xi )| \le |\xi|^{l} + \overline{g}(x),\q{with}\ \overline{g}\in L^q(\Om),\ \overline{g}_B\in L^{q_B}(\p\Om),\\
&\text{for}\quad  0<l< \tfrac{N+2}{N},\quad  q\ge 1, \q{and for }  q_B\ge 1.\nonumber
\end{align}

First, we collect a well-known elliptic regularity result for a nonhomogeneous  Neumann linear problem,
see \cite[Ch.3 Sec. 6]{Ladyzhenskaya_Ural’tseva},  \cite[Thm. 6.13]{Gilbarg_Trudinger},  \cite[Lem. 2.2]{Mavinga_Pardo} and \cite[Proposition B.1]{CMP_NonlA_2026}.

Let us consider the Neumann linear problem with reactions at the interior and on the boundary
\begin{equation}\label{eq:W1m:fx:fB}
\left \{
\begin{aligned}
-\Delta u +u=& \overline{g}(x), \; \; \; x\in \Om , \\
\frac{\p u}{\p \nu }=& \overline{g}_{B}(x), \;x\in \p \Om , 
\end{aligned}
\right.
\end{equation}
where
\begin{align}\label{overl:g}\tag{$\overline{g}.\overline{g}_B$}
\overline{g}\in L^q(\Om),\q{and } \overline{g}_B\in L^{q_B}(\p\Om),\text{ for } q\ge 1 \text{ and }  q_B\ge 1.
\end{align}

\begin{thm}\label{th:reg} Assume hypothesis \eqref{overl:g}.
If $\p\Om\in C^{0,1}$,   then there exists a unique $u\in W^{1,m}(\Om)$  solving \eqref{eq:W1m:fx:fB} and 
\begin{equation*}%\label{estim:W1m:fxfB:1}
\|u\|_{W^{1,m}(\Om)}\le C\left(\|\overline{g}\|_{L^{q}(\Om)}+\|\overline{g}_B\|_{L^{q_B}(\p\Om)}\right), 
\end{equation*} 
where $m=m(q,q_B)$ is given by
\begin{equation}\label{def:m:q}
m=\min\left\{q^*, \tfrac{Nq_B}{N-1} \right\}\ \text{if}\ q < N, \text{or}\ m=\min\left\{q,\tfrac{Nq_B}{N-1}\right\}\ \text{if}\  q \ge N,
\end{equation}  
and $C$ is independent of $u,\ \overline{g},\  \overline{g}_B$, and $q^*:= \tfrac{Nq}{N-q}.$ 
\end{thm}

\bigskip

Next, we observe that whenever $q_B>N-1,$ the previous regularity result can be improved in the following sense.

\begin{rem}\label{rem3} 
All throughout this remark, we assume $q_B\in (N-1,\infty)$. 	
\begin{enumerate}[label=\rm{(\roman*)},leftmargin=.2cm]
\item Observe  that, for all $q_B\in (N-1,\infty)$ and $N>2$, the following holds 
\begin{equation}\label{def:qB:tilde}
1<\widetilde{q_B}<N<\frac{Nq_B}{N-1}\qq{where} \widetilde{q_B}:=\frac{Nq_B}{N-1+q_B}.
\end{equation}
This allows us to define the following intervals $R_i=R_i(q_B)$ as
\begin{equation*}%\label{def:R}
\begin{array}{llll}
R_1:=\big[1, \widetilde{q_B}\big),
&R_2:=\big[ \widetilde{q_B}, N\big),
& R_3:=\Big[N, \frac{Nq_B}{N-1}\Big),
&R_4:=\Big[ \frac{Nq_B}{N-1}, \infty \Big).
\end{array}
\end{equation*}
Moreover, since \eqref{def:qB:tilde}, $R_i\cap R_j=\emptyset$ for $i\ne j$, $i,j=1,2,3,4$.

We observe that if $q<N$, then
\begin{equation}\label{tilq:q*}
q\,\lesseqqgtr\, \widetilde{q_B} \iff q^*\,\lesseqqgtr\, \frac{Nq_B}{N-1},
\end{equation}
\smallskip

\item With those definitions of $R_i$, for $m=m(q,q_B)$ defined by \eqref{def:m:q}, using \eqref{tilq:q*} the following holds

\begin{itemize}
\item If $q\in R_1$, then $m=q^*=\min\left\{q^*, \tfrac{Nq_B}{N-1} \right\};$ 
\item if $q\in R_2$, then $m=\frac{Nq_B}{N-1}=\min\left\{q^*, \tfrac{Nq_B}{N-1} \right\};$ 
\item if  $q\in R_3$, then $m=q<\frac{Nq_B}{N-1}$;
\item if $q\in R_4$, then $m=\frac{Nq_B}{N-1}=\min\left\{q^*, \tfrac{Nq_B}{N-1} \right\}.$
\end{itemize} 
\smallskip

\item Whenever $\ q\in R_3$, we can always choose any $\ \overline{q}\in  R_2$, $\overline{q}<q$, such that $\overline{g}\in L^{\overline{q}}(\Om)$,  and $(\overline{q})^*\ge\frac{Nq_B}{N-1}>q.$ Hence, the corresponding $m$ can be chosen as $m(\overline{q},q_B)=\frac{Nq_B}{N-1}>q=m(q,q_B).$ 
Redefining $m$ with the above criteria, we can  write
\begin{equation}\label{new:m}
m:=
\begin{cases}
q^* &\q{if} q<\widetilde{q_B}, \\
\frac{Nq_B}{N-1} &\q{if} q\ge \widetilde{q_B}.
\end{cases}
\end{equation}
\smallskip

\item For any $q>N/2$ and $q_B>N-1$, the inequality $m>N$  always holds for $m$ given by \eqref{new:m}.
\end{enumerate}
\end{rem}

Now, we rewrite Theorem \ref{th:reg} taking into account the Remark \ref{rem3}.(iii). 

\begin{thm}\label{th:reg:2}
Assume that the hypothesis of Theorem \ref{th:reg} holds. 
Assume also that $q_B\in (N-1,\infty)$. 
Let $u\in H^{\, 1}(\Om )$ be a weak solution to the linear problem \eqref{eq:W1m:fx:fB}.
Then, $u\in W^{1,m}(\Om )$, where  $m$ is defined by \eqref{new:m} and if $q>N/2$, then $m>N$. Moreover,    there exists a constant $C>0$  independent of $u$, $\overline{g}$ and $\overline{g}_{B}$ such that the following hold:
\smallskip

(i) If $\ q<\widetilde{q_B}$, then
\begin{equation*}%\label{estim:W1m:fxfB:1:2}
\|u\|_{W^{1,q^*}(\Om)}\le C\Big(\|\overline{g}\|_{L^{q}(\Om)}+\|\overline{g}_B\|_{L^{q_B}(\p\Om)}\Big), 
\end{equation*} 
\smallskip

(ii) If $\ q\ge \widetilde{q_B}$, then
\begin{equation*}%\label{estim:W1m:fxfB:1:tilde:ii}
\|u\|_{W^{1,\frac{Nq_B}{N-1}}(\Om)} \le C\Big(\|\overline{g}\|_{L^{\widetilde{q_B}}(\Om)} + \|\overline{g}_B\|_{L^{q_B}(\p\Om)}\Big).
\end{equation*} 
\end{thm}

\begin{rem}
Using the definition of $m$ \eqref{new:m}, the conclusions of Theorem \eqref{th:reg:2} can be rewritten in a unified way as follows
\begin{equation*}%\label{estim:min:q:qBtilde}
\|u\|_{W^{1,m}(\Om)} \le C\Big(\|\overline{g}\|_{L^{\min\{q,\widetilde{q_B}\}}(\Om)} + \|\overline{g}_B\|_{L^{q_B}(\p\Om)}\Big),
\end{equation*}
where $m$ is defined by \eqref{new:m}.
\end{rem}

Our purpose is now to estimate the $W^{1,m}(\Omega)$-norm of $u$ solving \eqref{eq:W1m:fx:Unif}, involving a power of the $\|\na u\|_{L^{2}(\Omega)}$-norm.

\begin{thm}\label{th:W1m:unif}
Let $u\in H^{\, 1}(\Om )$ be a weak solution to \eqref{eq:W1m:fx:Unif}
where $g,\ \overline{g}_B$ satisfy \eqref{g}, $\overline{g}\in L^q(\Om)$,  for  $ q>\tfrac{N}{2}$ and  $\overline{g}_B\in L^{q_B}(\p\Om)$ for $q_B> N-1$.\\
Then $u\in W^{1,m}(\Omega)$ where  $m=m(q,q_B)>N$ is defined by \eqref{new:m}, and there exists a constant $C=C(N, \Om , l, q, q_B)>0$, independent of $u,\overline{g},\overline{g}_B$, such that the following holds
\begin{equation}\label{estim:W1m:unificado}
\|u\|_{W^{1,m}(\Omega)} \le C\Big( \|\nabla u\|_{L^2(\Omega)}^{\gamma}
+ \|\overline{g}\|_{L^{\min\{q,\widetilde{q_B}\}}(\Omega)} + \|\overline{g}_B\|_{L^{q_B}(\partial\Omega)}\Big),
\end{equation}
where $\g=\g(N,l, q, q_B)>0$  is given as follows
\begin{equation}\label{Def:gamma}
 \begin{cases}
{\rm (I)}\ \ \ \gamma := l,\ \text{if } \widehat{q}<\widetilde{q_B}\text{ and } l\widehat{q}\le 2, &\text{or }  \widehat{q}\ge \widetilde{q_B}\text{ and }l\widetilde{q_B}\le 2,\\[2mm]
{\rm (II)}\ \ \gamma := \dfrac{2\,[\frac{N}{q}(1-l)+l]}{N(1-l)+2}, &\text{if } \; \widehat{q}\le \widetilde{q_B}\text{ and }l\widehat{q}> 2,\\[2mm]
{\rm (III)}\ \gamma := \dfrac{2\,[\frac{N-1}{q_B}(1-l)+1]}{N(1-l)+2},&\text{if } \; 
\widehat{q}\ge\widetilde{q_B}\text{ and }l\widetilde{q_B}> 2,
\end{cases} 
\end{equation}
for $\widetilde{q_B}$ given by \eqref{def:qB:tilde} and
\begin{equation}\label{def:q:hat}
\widehat{q}:=\min\Big\{\frac{m}{l},q\Big\}.
\end{equation}
\end{thm}

\begin{proof}
The proof is based on a four-step regularizing procedure, proving that $u\in W^{1,m}(\Om )$, and a fifth step to reach  the estimate \eqref{estim:W1m:unificado}.   
We assume throughout all the proof that $0< l< \tfrac{N+2}{N}$. 
\\

We use a bootstrap procedure based on Theorem \ref{th:reg:2}. Given $m_0:=2$, for $u\in W^{1,m_{j}}(\Om)$, $j\ge 0$,  defining, 
\begin{equation}\label{def:qj}
q_j:=\min \left\{ \frac{m_j}{l}, q\right\}, \qq{so}
|\nabla u|^l + \overline{g} \in L^{q_j}(\Om),
\end{equation}
and 
\begin{equation}\label{def:mj}
m_{j+1}:=\left.
\begin{cases}
q_j^*, & \text{if } q_j<\widetilde{q_B}\\[1mm]
\frac{Nq_B}{N-1}, & \text{if } q_j\ge\widetilde{q_B} 
\end{cases} \right\}, \q{then} u\in W^{1,m_{j+1}}(\Om).
\end{equation} 
In a finite number of iterations, we prove that,  $m_j\ge m$, and  $u\in W^{1,m}(\Om )$. We reason by contradiction, assuming that $\{q_j\},\ \{m_j\}$ are sequences with infinitely many  elements and that $m_{j}<m$ for all $j$.
\bigskip

In Step 1 we prove  that if there exists $j_0\in \mathbb{N}$ such that $q_{j_0}\ge q$, then  $u\in W^{1,m}(\Om )$ and this part of the proof is finished.	

Assume, on the contrary, that  the following holds 
\begin{equation}\label{a1}\tag{\bf{a.1}}
q_j =\frac{m_j}{l}<q, \quad \text{for all }\quad j\in \N.
\end{equation}

Under hypotheses \eqref{a1}, we prove in step 2  that if there exists  $j_0\in \mathbb{N}$ such that $q_{j_0}\ge \widetilde{q_B}$, then $m_{j_0 +1}\ge m$,  and $u\in W^{1,m}(\Om )$.

Assume then 
\begin{equation}\label{a2}\tag{\bf{a.2}}
q_j<\widetilde{q_B}, \ \text{for all}\ j\in \N\implies m_{j+1}=q_j^*, \ \text{for all}\ j\in \N.
\end{equation} 

In step 3, under hypotheses \eqref{a1}--\eqref{a2},  we conclude that, 
\begin{equation}\label{qj:2N}
\{ q_j\}_{j\ge 1} \subset  \big(\tfrac{2N}{N+2},\widetilde{q_B}\big), \q{and }\{ q_j\}_{j\ge 1} \text{ is increasing}.
\end{equation}

Step 4 is devoted to prove that there are no increasing sequences $\{q_j\}_{j=1}^\infty\subset  \big(\frac{2N}{N+2},\widetilde{q_B}\big) $, concluding that $u\in W^{1,m}(\Om )$. 

Let us carry out the procedure.
\medskip

\paragraph{\bf{Step 1}} {\it 
If there exists $j_0\in \mathbb{N}$ such that $q_{j_0}\ge q$, then $u\in W^{1,m}(\Om )$.	}
\medskip

\noindent  
Indeed, assume on the one hand that $ q\ge \widetilde{q_B}$.  
Then  $q_{j_0}\ge q\ge \widetilde{q_B}$ and $m_{j_0+1}= \frac{Nq_B}{N-1}= m$, see \eqref{new:m}. \\
\noindent   Assume on the other hand that $q<\widetilde{q_B}$, so $m=q^*$ (see \eqref{new:m}).  Obviously, if $q_{j_0}\ge q$, then $|\nabla u|^l + \overline{g} \in L^{q}(\Om)$ (see \eqref{def:qj}), and, by Theorem \ref{th:reg:2},  $\ u\in W^{1,q^*}(\Om)$.
In both cases, the proof of Step 1 is achieved.

\bigskip

\paragraph{\bf{Step 2}} 
{\it Assume hypothesis \eqref{a1}. If there exists  $j_0\in \mathbb{N}$ such that  $q_{j_0}\ge\widetilde{q_B}$,
then $m_{j_0 +1}=\frac{Nq_B}{N-1}\ge m$, and  $u\in W^{1,m}(\Om )$.	
}

\medskip

\noindent Assume that $q_{j_0}\ge\widetilde{q_B}$,
by \eqref{def:mj}, we  deduce that $m_{j_0+1}=\frac{Nq_B}{N-1}\ge m$. 

\bigskip

\paragraph{\bf{Step 3}} {\it Assume hypothesis \eqref{a1} and \eqref{a2}. Then, \eqref{qj:2N} holds.
}

\medskip

\noindent First, in the following claim we analyze the monotonicity of the sequence $\{q_j\}_{j\ge 0}.$
\medskip

{\it Claim 1: Assume  hypothesis \eqref{a1} and \eqref{a2}. Then}
\begin{equation}\label{1:l:monot}
q_{j+1}> q_j \iff q_j > N\left( 1-\tfrac{1}{l}\right).    
\end{equation}
Indeed, observe that $q_j<\widetilde{q_B}\iff \frac{N-1}{Nq_B}+\frac{1}{N}<\frac{1}{q_j}$, since  \eqref{def:qB:tilde}, $q_j<N$,  then  $\frac{N-1}{Nq_B}<\frac{1}{q_j^*}$, using \eqref{def:mj} $m_{j+1}=q_j^*$, and using \eqref{a1}, $q_{j+1}=q_j^*/l$. \\
Obviously, $q_{j+1}> q_j $ if and only if $  \frac{1}{q_j}> l\big( \frac{1}{q_j}-\frac{1}{N}\big)$ or equivalently $ \frac{1}{N}> \left( 1-\frac{1}{l}\right)\frac{1}{q_j}$, in other words $q_j > N\left( 1-\frac{1}{l}\right)$, ending the proof of Claim 1.

\bigskip

\noindent Now, we go back to the proof of Step 3. Since \eqref{a1} we observe that 
\begin{equation}\label{q0:2N}
q_0=\tfrac{2}{l}>\tfrac{2N}{N+2}>1 \q{for all } l<\tfrac{N+2}{N},
\end{equation}
and that
\begin{align}	
\label{l:12N}
&\tfrac{2N}{N+2}> N\left( 1-\tfrac{1}{l}\right)\iff l< \tfrac{N+2}{N} \iff\tfrac{2}{l}> N\left( 1-\tfrac{1}{l}\right).
\end{align}
Using \eqref{q0:2N} and \eqref{l:12N}, we deduce that 
\begin{equation}\label{q0:}
q_0>\tfrac{2N}{N+2}> N\left( 1-\tfrac{1}{l}\right).   
\end{equation}
Since \eqref{q0:} and \eqref{a2}, $q_0\in R_1\cap \big(\frac{2N}{N+2},\widetilde{q_B}\big),$
and by Claim 1, \eqref{1:l:monot}, and \eqref{q0:},  $q_j$ is increasing.  Finally, using hypothesis \eqref{a2}, $q_j$ satisfies \eqref{qj:2N} for all $j\ge 0$, ending the  proof of Step 3.

\bigskip

\paragraph{\bf{Step 4}} {\it Assume hypothesis \eqref{a1}--\eqref{a2}. 
Then, there are not increasing sequences 
with infinitely many elements 
$\{q_j\}_{j=0}^\infty\subset  \big(\frac{2N}{N+2},\widetilde{q_B}\big) $.}

\medskip

\noindent We  proceed by contradiction, using step 3, there exists an  increasing sequence, $\{q_j\}\subset  \big(\frac{2N}{N+2},\widetilde{q_B}\big) $.
Then, there exists 
\begin{equation}\label{Lim:2}
\lim _{j\to \infty }q_{j}=:q_{\infty} \in \big[\tfrac{2N}{N+2},\widetilde{q_B}\big].
\end{equation} 
Using \eqref{a2}, \eqref{a1}$_{j+1}$,  and taking limits
$
\tfrac{1}{q_\infty}= l\big( \tfrac{1}{q_\infty}-\tfrac{1}{N}\big)\iff q_\infty =N\left( 1-\tfrac{1}{l}\right),
$
which  contradicts  \eqref{Lim:2} and \eqref{l:12N}.
\bigskip

As a conclusion of  Steps 1 --  4, in a finite number of iterations, $u\in W^{1,m}(\Om )$ where  $m$ is defined by \eqref{new:m}.	
\bigskip

\paragraph{\bf Step 5}{\it The estimate \eqref{estim:W1m:unificado} holds.}
\medskip

\noindent Since $\ u\in W^{1,m}(\Om)$, then
$
\ |\nabla u|^l  \in L^{m/l}(\Om)\q{with}\|\,|\nabla u|^l\|_{L^{m/l}(\Omega)}= 
\|\nabla u\|_{L^{m}(\Omega)}^{\,l},\;\, 
$
and $|\nabla u|^l + \overline{g} \in L^{\widehat{q}}(\Om),\ $ for  $\widehat{q}$ defined by \eqref{def:q:hat}. Using Theorem \ref{th:reg:2}, the following holds:

(i) If $\ \widehat{q}<\widetilde{q_B}$, then $u\in W^{1,(\widehat{q})^*}(\Om )$, and 
\begin{equation}\label{estim:W1m:fxfB:i1} 
\|u\|_{W^{1,(\widehat{q})^*}(\Om)}\le C\Big(\| |\na u|\|^l_{L^{l\widehat{q}} (\Om)}+\|\overline{g}\|_{L^{\widehat{q}} (\Om)} +\|\overline{g}_B\|_{L^{q_B}(\p\Om)}\Big). 
\end{equation} 

(ii) If $\ \widehat{q}\ge \widetilde{q_B}$, then $u\in W^{1,\frac{Nq_B}{N-1}}(\Om )$, and 
\begin{equation}\label{estim:W1m:fxfB:1:tilde:ii2}  
\|u\|_{W^{1,\frac{Nq_B}{N-1}}(\Om)}\le C\Big(
\| |\na u|\|^l_{L^{l\widetilde{q_B}} (\Om)}
+\|\overline{g}\|_{L^{\widetilde{q_B}} (\Om)}
+\|\overline{g}_B\|_{L^{q_B}(\p\Om)}\Big), 
\end{equation} 
where $\widetilde{q_B}$ is defined in \eqref{def:qB:tilde}.

\medskip

To obtain estimate \eqref{estim:W1m:unificado}, we proceed as follows. We next prove in Claim 2 that $(\widehat{q})^*= m$. In Claim 3 we prove that the estimate \eqref{estim:W1m:unificado} holds when $\ l\,\min\{\widehat{q},l\widetilde{q_B}\}\le 2$.
In Claim 4 we analyze the case when $l\widehat{q}> 2$, or $l\widetilde{q_B}> 2$, using a convenient interpolation inequality.

\medskip
 
\noindent Assume on the one hand that $\widehat{q}<\widetilde{q_B}$. 
By definition, see  \eqref{def:q:hat}, either $\widehat{q}=\frac{m}{l}$ 
or $\widehat{q}=q$. Once proved Claim 2, we have that $(\widehat{q})^*= m$.\\
On the other hand, if $\widehat{q}\ge \widetilde{q_B}$, then $q\ge \widetilde{q_B}$ and $m=\frac{Nq_B}{N-1} $, see \eqref{new:m}.\\
In both cases, the  RHS of the estimates \eqref{estim:W1m:fxfB:i1}--\eqref{estim:W1m:fxfB:1:tilde:ii2}  are estimates of the $W^{1,{m}}(\Om )$-norm.

\medskip

{\it Claim 2: Assume $\ \widehat{q}<\widetilde{q_B}$, case {\rm (i)}, then 
} 
\begin{align*}
&\widehat{q}=q<\widetilde{q_B},\qq{and }
u\in W^{1,(\widehat{q})^*}(\Om )\text{ with }  (\widehat{q})^*= m=q^*.
\end{align*} 
First, we observe that if  $\ \widehat{q}<\widetilde{q_B}$, then the following holds:
\begin{enumerate}[label=\rm{C$_2$.\arabic*)},leftmargin=.2cm]
\item 
\label{m:q*}
$q<\widetilde{q_B}$ and $m=q^*$;
\item %\label{C2.2} 
\label{hat:q*}
$(\widehat{q})^*\ge m ; $ 
\item \label{C2.3}
$\widehat{q}=q.$
\end{enumerate}
Indeed, either $\widehat{q}=q <\frac{m}{l}$ or $\widehat{q}=\frac{m}{l}\le q$. 
Obviously, if $\widehat{q}=q<\widetilde{q_B}$, then $(\widehat{q})^*=q^*=m$ and so \ref{m:q*}-\ref{hat:q*}-\ref{C2.3} hold in that case. \\
Assume now that $\widehat{q}=\frac{m}{l}<\widetilde{q_B}.$ 

\begin{enumerate}[label=\rm{C$_2$.\arabic*)}, leftmargin=.2cm]
\item 
Assume on the contrary that \ref{m:q*} do not hold, so $ q\ge \widetilde{q_B}$ and $m=\frac{Nq_B}{N-1}$. Then,
$\widehat{q}=\frac{m}{l}<\widetilde{q_B}\iff
\frac{l}{m}<\frac{1}{\widetilde{q_B}}\iff \frac{l(N-1)}{Nq_B}<\frac{N-1}{Nq_B}+\frac{1}{N}\iff (l-1)(N-1)<q_B.$ Besides, $(l-1)(N-1)<\frac{2}{N}(N-1)<(N-1)<q_B,$ hence all those inequalities hold, implying that $\widehat{q}>\widetilde{q_B},$ which contradicts hypothesis $ \widehat{q}<\widetilde{q_B}$ in Claim 2.  Hence,   \ref{m:q*} hold.

\item Assume that $\widehat{q}=\frac{m}{l}<\widetilde{q_B}<N,$ see \eqref{def:qB:tilde}. Observe that $(\widehat{q})^*=(\frac{m}{l})^*> m$  if and only if $\frac{l}{m}-\frac{1}{N}<\frac{1}{m}$, which in turn is equivalent to $\frac{l-1}{m} <\frac{1}{N}$, or $m> (l-1)N$, which easily holds since $m>N>2=\frac{2}{N}N>(l-1)N$, where we have used Remark \ref{rem3}(iv) ensuring that $m>N$. Hence, \ref{hat:q*} hold.
\end{enumerate}

\begin{enumerate}[label=\rm{C$_2$.3)}, leftmargin=.2cm]
\item Using the first equivalence in \eqref{l:12N}, and that $\tfrac{N}{2}< q \le \widehat{q}<\widetilde{q_B}<N,$ we can write
\begin{align}\label{q:N2}
&q>\tfrac{N}{2}>\tfrac{2N}{N+2}>N\big(1-\tfrac{1}{l}\big)\implies \tfrac{1}{N}>\tfrac{1}{q}\big(1-\tfrac{1}{l}\big),\\
\label{q:N2:2}
&\text{and }\ \tfrac{1}{N}>\tfrac{1}{q}\big(1-\tfrac{1}{l}\big)\iff \tfrac{1}{ql}>\tfrac{1}{q}-\tfrac{1}{N}\iff q^*>ql.
\end{align}
Besides, using \ref{m:q*}, it turns out that $m=q^*,$  so  $m>lq$, $\ \widehat{q}=q$ and \ref{C2.3} hold.
\end{enumerate}
\bigskip

{\it Claim 3: Assume that 
\begin{align*}%\label{lq2:0}
\text{either}\quad l\widehat{q}\le 2& \qq{when}  \widehat{q}<\widetilde{q_B},   \\
\text{or}\qquad l\widetilde{q_B}\le 2 &\qq{when}   \widehat{q}\ge\widetilde{q_B}   ,\nonumber
\end{align*}
for $\widetilde{q_B}$ defined by \eqref{def:qB:tilde}. Then,  the estimate \eqref{estim:W1m:unificado} is achieved.}
\medskip

\noindent Indeed, if $\ l\widehat{q}\le 2$, then
$
L^{2}(\Omega)\hookrightarrow L^{l\widehat{q}}(\Omega)
$
and
\begin{equation}\label{power-L2}
\|\nabla u\|_{L^{l\widehat{q}}(\Om)}
\le |\Omega|^{\, \frac{1}{l\widehat{q}}-\frac{1}{2}}\,
\|\nabla u\|_{L^{2}(\Om)}.
\end{equation}
Substituting \eqref{power-L2} into \eqref{estim:W1m:fxfB:i1}, and using Claim  2, for $\ \widehat{q}<\widetilde{q_B}$ we get 
\begin{equation*}%\label{bound-with-L2}
\|u\|_{W^{1,(\widehat{q})^*}(\Om)} \le C\Big(\|\nabla u\|_{L^{2}(\Om)}^{\,l}
+ \|\overline g\|_{L^{\widehat{q}}(\Om)} + \|\overline{g}_B\|_{L^{q_B}(\p\Om)}\Big).
\end{equation*}
Likewise, if $\ l\widetilde{q_B}\le 2$,  then
$
L^{2}(\Omega)\hookrightarrow L^{l\widetilde{q_B}}(\Omega)
$
and
\begin{equation}\label{power-L2:2}
\|\nabla u\|_{L^{l\widetilde{q_B}}(\Om)}
\le |\Omega|^{\,\frac{1}{l\widetilde{q_B}}-\frac{1}{2}}\,
\|\nabla u\|_{L^{2}(\Om)}.
\end{equation}
Substituting \eqref{power-L2:2} into  \eqref{estim:W1m:fxfB:1:tilde:ii2}, for $\ \widehat{q}\ge \widetilde{q_B}$ we get
\begin{equation*}%\label{estim:W1m:fxfB:i112} 
\|u\|_{W^{1,\frac{Nq_B}{N-1}}(\Om)}\le C\Big(\|\nabla u\|_{L^{2}(\Om)}^{\,l}
+ \|\overline g\|_{L^{\widetilde{q_B}}(\Om)} + \|\overline{g}_B\|_{L^{q_B}(\p\Om)}\Big),
\end{equation*}
ending  the proof of Claim 3.
\bigskip

{\it Claim 4: Assume that 
\begin{align}\label{lq2}
\text{either}\quad l\widehat{q}> 2& \qq{when}  \widehat{q}<\widetilde{q_B},  \\
\text{or}\qquad l\widetilde{q_B}>2&\qq{when}\widehat{q}\ge\widetilde{q_B}.\nonumber
\end{align}
for $\widetilde{q_B}$ defined by \eqref{def:qB:tilde}. Then,  the estimate \eqref{estim:W1m:unificado} is achieved.}
\medskip

\noindent From now on, we use Lebesgue interpolation inequalities, see  \cite[Remark 2, p. 93]{Brezis}.
To treat  the estimates \eqref{estim:W1m:fxfB:i1}-\eqref{estim:W1m:fxfB:1:tilde:ii2} in an unified way, using  Claim  2 when $\widehat{q}<\widetilde{q_B},$ we  define $(\eta,\kappa )$, in terms of $N,l,q,q_B$:
\begin{align}\label{def:eta:kap}
(\eta, \kappa) =\begin{cases}
{\rm a)}\ (q^*,lq) & \text{if } \; \widehat{q}=q<\widetilde{q_B},\\[1mm]
{\rm b)}\ \big(\frac{Nq_B}{N-1},\frac{l\, Nq_B}{N-1+q_B}\big) & \text{if } \; \widehat{q}\ge \widetilde{q_B}.
\end{cases}
\end{align}
Later we will check that the following condition holds:
\begin{equation}\label{kap:eta} \tag{c.1}
2<\kappa <\eta .    
\end{equation} 
Under condition \eqref{kap:eta}, the interpolation inequality  states that
\begin{equation}\label{Interp:grad}
\|\nabla u\|_{L^{\kappa}(\Om)}\le \|\nabla u\|_{L^{2}(\Om)}^{\,\Theta}\,\|\na u\|_{L^{\eta}(\Om)}^{\,1-\Theta} \le \|\nabla u\|_{L^{2}(\Om)}^{\,\Theta}\,\|u\|_{W^{1,\eta}(\Om)}^{\,1-\Theta},
\end{equation}
where 
\begin{equation*}%\label{Def:theta}
\frac{1}{\kappa}=\frac{\Theta}{2}+\frac{1-\Theta}{\eta } \iff \Theta \;=\; \frac{\frac{1}{\kappa} - \frac{1}{\eta }}{\frac{1}{2} - \frac{1}{\eta}}\in (0,1).
\end{equation*}

Assuming also the following condition: 
\begin{equation}\label{l:te:1}\tag{c.2}
l(1-\Te)<1 \iff l\Big( \frac{1}{2}-\frac{1}{\kappa}\Big)< \frac{1}{2}- \frac{1}{\eta},
\end{equation}
we choose
\begin{equation}\label{al:al'}
\La=\frac{1}{\,l(1-\Theta)\,},\q{and } \La'=\frac{1}{\,1-l(1-\Theta)\,}>0.
\end{equation}
Using the Young's inequality, see \cite[Section 7.1, p.145]{Gilbarg_Trudinger}, into \eqref{Interp:grad} and rising to the power $l$, we obtain the following:  for all $\e>0$,
\begin{equation}\label{Des:Young:1}
\|\nabla u\|_{L^{\kappa}(\Om)}^{\,l} \le \frac{\varepsilon}{\La}\,\|u\|_{W^{1,\eta}(\Om)}
+ C(\varepsilon,\La)\,\|\nabla u\|_{L^{2}(\Om)}^{\,\gamma},
\end{equation}
where $ C(\varepsilon,\La)=\frac{\varepsilon^{-\La'/\La}}{\La'}$ and
\begin{equation}\label{gamma-def:0}
\gamma = l\Theta \La'>0, \qq{under condition} {\rm\eqref{l:te:1}}, 
\end{equation}
since $l>0$, $\Te\in(0,1)$,  and \eqref{al:al'}.\\
Now,  fixing $\e$  such that  $C\frac{\e}{\La}=\frac{1}{2}$ and substituting 
\eqref{Des:Young:1} into \eqref{estim:W1m:fxfB:i1} or into \eqref{estim:W1m:fxfB:1:tilde:ii2}, respectively, we deduce
\begin{equation*}%\label{estim:eta}
\|u\|_{W^{1,{\eta }}(\Om )}
\leq C\Big(\|\nabla u\|_{L^{2}(\Om)}^{\g} + \| \overline{g}\|_{L^{\min\{\widetilde{q_B},\widehat{q}\}}(\Om)} 
+ \| \overline{g}_B\|_{L^{q_B}(\p \Om)} \Big),
\end{equation*}
where $\eta,\ \gamma ,\  \widetilde{q_B}$ and $\widehat{q}$ are given by \eqref{def:eta:kap}, \eqref{Def:gamma},  \eqref{def:qB:tilde}, and \eqref{def:q:hat} respectively. Using the definition of $\eta$,   the definition of $m$, see \eqref{new:m}, and Claim 2, 
$\eta=m=q^*$  in case a) and $\eta=m=\frac{Nq_B}{N-1}$  in case b). We now observe that this estimate is precisely estimate \eqref{estim:W1m:unificado}.

To end the proof of this Step, we only have to verify that for the prior choices of $\kappa$ and $\eta$, conditions \eqref{kap:eta} and \eqref{l:te:1} hold.
Indeed, since 
\eqref{lq2},  it follows that $\kappa > 2.$ 
Next, for cases a) and b), we  prove that  $\eta>\kappa$ in the following claims 5 and 6.

\noindent a) {\it Claim 5: If $q<\widetilde{q_B},$ then $\eta= q^*>lq=\kappa$ for all $0<l<\tfrac{N+2}{N}$.}\\
\noindent Indeed, if $q<\widetilde{q_B},$ then $m=q^*$, and using  
\eqref{q:N2}-\eqref{q:N2:2} we  check that $q^*>lq$, so $\eta>\kappa$.

\noindent b) {\it Claim 6: If $q\ge \widetilde{q_B},$ then: $\eta=\frac{Nq_B}{N-1}>\frac{l\, Nq_B}{N-1+q_B}=\kappa$.}\\
\noindent We prove that $\frac{1}{N-1}>\frac{l}{N-1+q_B}$.
Equivalently, we check that $q_B>(l-1)(N-1)$.
Since $l<1+\tfrac{2}{N}$, $N>2$ and by hypothesis \(q_B>N-1\), we have 
$(l-1)(N-1)<\frac{2(N-1)}{N}<N-1<q_B$,   concluding that   $\eta>\kappa$.

In addition, we verify that \eqref{l:te:1} is fulfilled in  cases a), b) above.\\
a) We observe that $\eta=m=q^* $, $\kappa =lq,$ so 
\eqref{l:te:1} is reduced to
$
l\big(\tfrac{1}{2}-\tfrac{1}{lq}\big)<\tfrac{1}{2}-\tfrac{1}{q}+\tfrac{1}{N}\iff\tfrac{l-1}{2}<\tfrac{1}{N} \,.
$\\
\noindent b)  Now we have that $\eta=\frac{Nq_B}{N-1}$ and $\kappa=\frac{l\, Nq_B}{N-1+q_B}$. Therefore,
$
l\big( \tfrac{1}{2}- \tfrac{N-1+q_B}{lNq_B}\big)< \tfrac{1}{2} - \tfrac{N-1}{Nq_B}  \iff  \tfrac{l-1}{2}<\tfrac{1}{N}\,,
$
and so \eqref{l:te:1} holds in both cases, a) and b).
This ends the proof of Step 5 cases (i) and (ii), and consequently also ends the proof of Theorem \ref{th:W1m:unif}. 
\end{proof}

\begin{rem}
Using \eqref{gamma-def:0},  one definition  of $\g$ equivalent to \eqref{Def:gamma}, 
\begin{equation*}%\label{Des:gamma:pos}
\g \,\lesseqqgtr\, 1 \iff   l\Te \,\lesseqqgtr\, \frac{1}{\La'} \iff l\Te \,\lesseqqgtr\, 1-l(1-\Te) \iff l\,\lesseqqgtr\, 1.
\end{equation*}
\end{rem}

The next Lemma provides a bound on the $H^1-$norm of solutions
to \eqref{eq:W1m:fx:Unif}, in terms of the critical Lebesgue and trace norms.

\begin{lem}\label{lem:H1:bd}
Let $u\in H^{\, 1}(\Om )$ be a weak solution to \eqref{eq:W1m:fx:Unif}
where $g,\ \overline{g}_B$ satisfy \eqref{g},  for  $ q\ge(2^*)'$ and  $q_B\ge (2_*)'.$  %(see Definition \eqref{Conj:Sob:trac}).

Then,  there exists a constant $C=C(N, \Om , l, q, q_B)>0$, independent of $u,\overline{g},\overline{g}_B$, such that the following holds
\begin{align}\label{eq:H1:bd}
\|u\|_{H^{1}(\Omega)}^2 \le C\Big( \|u\|_{L^{2^*}(\Omega)}^{\frac{2 }{2-l}}&+ \|\overline{g}\|_{L^{(2^*)'}(\Omega)} \|u\|_{L^{2^*}(\Omega)}\nonumber\\
& + \|\overline{g}_B\|_{L^{(2_*)'}(\partial\Omega)}\|u\|_{L^{2_*}(\p\Omega)}\Big).
\end{align}
\end{lem}

\begin{proof}%[Proof of Lemma \ref{lem:H1:bd}]
Using the  definition of a weak solution, hypothesis (\ref{g}) for $q\ge(2^*)'$, $q_B\ge (2_*)'$,
and Hölder inequality
\begin{align*}
&\int_\Om |\na u|^2 + u^2  \le  \int_\Om |\na u|^l\,|u| 
 + \int_\Om \overline{g}(x)\,|u|+\int_{\p\Om} |\overline{g}_{B}(x)|\,|u|\\
& \qquad  \le  \left(\int_\Om|\na u|^2\right)^\frac{l}{2}\left(\int_\Om|u|^\frac{2}{2-l}\right)^\frac{2-l}{2}+\left(\int_\Om|\overline{g}(x)|^{(2^*)'}\right)^{\frac1{(2^*)'}}\left(\int_\Om|u|^{2^*}\right)^{\frac1{2^*}} \\
& \qquad\qquad+\left(\int_{\p\Om}|\overline{g}_{B}(x)|^{(2_*)'}\right)^{\frac1{(2_*)'}}\left(\int_{\p\Om}|u|^{2_*}\right)^{\frac1{2_*}},
\end{align*}
in other words
\begin{align}\label{H1:1}
\|u|_{H^1(\Om)}^2 
&\le \|\na u\|_{L^2(\Om)}^l \|u\|_{L^\frac{2}{2-l}(\Om)}
+\|\overline{g}\|_{L^{(2^*)'(\Om)}} \| u\|_{L^{2^*}(\Om)} \nonumber\\
&\qquad\qquad+\|\overline{g}_B\|_{L^{(2_*)'(\p\Om)}} \| u\|_{L^{2_*}(\p\Om)}.
\end{align}
Now, using Young's inequality, see \cite[Section 7.1, p.145]{Gilbarg_Trudinger}, and since $\frac{2}{2-l}<2^*,$ we deduce that there exists a constant $C=C(l,|\Om|)>0$ and independent of $u$, such that
\begin{align}\label{H1:2}
\|\na u\|_{L^2(\Om)}^l \|u\|_{L^\frac{2}{2-l}(\Om)}
\le \frac12 \|\na u\|_{L^2(\Om)}^2+ C\| u\|_{L^{2^*}(\Om)}^\frac{2}{2-l}.
\end{align}
Hence, substituting \eqref{H1:2} into \eqref{H1:1}, we obtain the estimate \eqref{eq:H1:bd}, ending the proof of Lemma \ref{lem:H1:bd}.
\end{proof}

The next Proposition establishes an \textit{a priori} bound on the $W^{1,m}-$norm of solutions
to \eqref{eq:W1m:fx:Unif}, expressed in terms of the Lebesgue norms of the solution and of the nonlinearities $\overline{g}$ and $\overline{g}_B$.

\begin{pro}
Assume the same hypothesis of the Theorem \eqref{th:W1m:unif}. \\
Then there exists a constant $C=C(N, \Om , l, q, q_B)>0$, independent of $u,\overline{g},\overline{g}_B$, such that the following holds
\begin{align*}%\label{estim:W1m}
&\|u\|_{W^{1,m}(\Omega)} 
\le C\Big( 1+ \|u\|_{L^{2^*}(\Omega)}^{\frac{\g }{2-l}}
+ \|\overline{g}\|_{L^{(2^*)'}(\Omega)}^\frac{\g}{2} \|u\|_{L^{2^*}(\Omega)}^\frac{\g}{2}\\
&\quad 
+ \|\overline{g}_B\|_{L^{(2_*)'}(\p\Om)}^\frac{\g}{2}
\|u\|_{L^{2_*}(\p\Omega)}^\frac{\g}{2}\nonumber 
+ \|\overline{g}\|_{L^{\min\{q,\widetilde{q_B}\}}(\Omega)} + \|\overline{g}_B\|_{L^{q_B}(\partial\Omega)}\Big),
\end{align*}
where $\g=\g(N,l, q, q_B)>0$  is given in \eqref{Def:gamma}.
\end{pro}

\begin{proof}
We only have to observe that $N/2>2N/(N+2)$,   that $N-1>2(N-1)/N$, and to introduce estimate \eqref{eq:H1:bd} into estimate \eqref{estim:W1m:unificado}.
\end{proof}

\section{Proof of Theorem \ref{th:RegEst} and of Corollary \ref{cor1}}
\label{sec:expl:estim}

\begin{proof}[Proof of Theorem \ref{th:RegEst}]
Let $ u \subset H^{1}(\Om )$ be a sequence of weak solutions to \eqref{pde}. 
Using Theorem \ref{th:reg:nol}, $u \subset H^{1}(\Om )\cap L^{\infty } (\Om ),$
hence $\widehat{f}(\cdot,u(\cdot))\in L^{r} (\Om) ,$ where 
$\widehat{f}$ is defined in \eqref{f:hat}.
We establish the estimate \eqref{estim} in several steps below.
\medskip

\noindent{(i) \it  $W^{1,m}(\Om)$ estimates of $u$, for $m>N$.}\\
First,  using elliptic regularity theory, specifically Theorem \ref{th:W1m:unif}, for
\begin{equation}\label{q:qB}
	q\in\left(\tfrac{N}{2},\min\left\{r,\widetilde{r_B}\right\}\right), \quad 
	q_B :=\tfrac{(N-1)q^*}{N}\in (N-1,r_B),
\end{equation}
where 
\begin{equation}\label{def:rB:tilde}
\widetilde{r_B}:=\tfrac{Nr_B}{N-1+r_B}<N,
\end{equation}
we deduce that $u\in W^{1,m}(\Omega)$ where  $m=m(q,q_B)$ is defined by  \eqref{new:m}.
Observe that with those definitions of $q$ and $ q_B$,  
\begin{equation}\label{def:m:q:2}
	q = \widetilde{q_B}\qq{and} m=q^* =\tfrac{Nq_B}{N-1}>N.
\end{equation}
Using \eqref{q:N2}-\eqref{q:N2:2}, $m/l>q$ and by definition of $\widehat{q}$ and $\widetilde{q_B}$, see \eqref{def:q:hat} and \eqref{def:qB:tilde} respectively,
\begin{equation*}%\label{def:m:q:4}
	\widehat{q}=q=\widetilde{q_B}.
\end{equation*}
Moreover, we have the following equivalences
\begin{equation}\label{equiv:q:qB}
	q_B:=\tfrac{(N-1)q^*}{N}\iff \tfrac{2_*}{q_B}=\tfrac{2^*}{q^*}\iff 2_{*,N/q_B} =  2^*_{N/q}.
\end{equation}

Additionally,  using \eqref{estim:W1m:unificado},  there exists a constant $C=C(N, \Om , l, q)>0$, independent of $u,\widehat{f},f_B$, such that the following holds
\begin{align}\label{estim:W1m:unificado:q}
\|u\|_{W^{1,m}(\Omega)} \le & C\Big( \|\nabla u\|_{L^2(\Omega)}^{\gamma}+ \|\widehat{f}(\cdot,u(\cdot))\|_{L^{q}(\Omega)} 
\nonumber\\
&\qquad\qquad  
+ \|f_B(\cdot,u(\cdot))\|_{L^{q_B}(\partial\Omega)}\Big),
\end{align}
and adapting 
\eqref{Def:gamma} for this particular case,
\begin{equation}\label{Def:gamma:2}
\gamma :=  
\begin{cases}
{\rm (a)}\ \, l,\ \text{if } lq\le 2,\\[3mm]
{\rm (b)}\ \dfrac{2\,[\frac{N}{q}(1-l)+l]}{N(1-l)+2}
,&\text{if }  lq> 2.
\end{cases}
\end{equation}
Observe that whenever $lq> 2,$ then
$\frac{2\,[\frac{N}{q}(1-l)+l]}{N(1-l)+2}\lesseqqgtr l\iff l\lesseqqgtr1.$\\

\noindent{(ii) \it   Gagliardo-Nirenberg interpolation inequality.}\\
We now use the Gagliardo--Nirenberg interpolation inequality
\begin{equation}\label{GN}
\|u\|_{L^{\infty}(\Omega)}
\le C\,\|u\|_{W^{1,m}(\Omega)}^{\sigma}\,\|u\|_{L^{2^*}(\Omega)}^{1-\sigma},
\end{equation}
where $\sigma\in(0,1]$ is given by 
\begin{equation*}
\tfrac{1-\sigma}{\sigma}=2^*\big(\tfrac{1}{N}-\tfrac{1}{m}\big),
\qquad
\text{that is, }
\qquad
\sigma:=\tfrac{1}{1+2^*\big(\tfrac{1}{N}-\tfrac{1}{m}\big)}.
\end{equation*}
Since \eqref{def:m:q:2}, $m=q^*$ and we have 
\begin{equation}\label{Def:sigma:2}
\tfrac{1}{\s}
=1+2^*\big(\tfrac{2}{N}-\tfrac{1}{q}\big)
= 2^*_{N/q}-1.
\end{equation}

We introduce the  $W^{1,m}(\Om)$--estimate \eqref{estim:W1m:unificado:q}, into the Gagliardo--Nirenberg interpolation inequality \eqref{GN} to obtain
\begin{equation}\label{step1}
\|u\|_{L^{\infty}(\Om)}
\le C\Big( 
\|\nabla u\|_{L^2(\Om)}^{\gamma}
+\|\widehat{f}(\cdot,u)\|_{L^{q}(\Om)}
+\|f_B(\cdot,u)\|_{L^{q_B}(\p\Om)}
\Big)^{\sigma}
\|u\|_{L^{2^*}(\Om)}^{1-\sigma},
\end{equation}
where $q,\ q_B$ are defined in \eqref{q:qB}
and
$\g=\g(N,l, q)$ is defined by \eqref{Def:gamma:2}.
\bigskip

\noindent{(iii) \it  Estimates in terms of the Lebesgue norms.}\\
Using that $\widetilde{f}$ and $\widetilde{f}_B$ are non-decreasing, we define 
\begin{eqnarray}\label{def:M:MB}
M&:=& \widetilde{f}(\|u\| _{L^{\infty }(\Om )})=\max _{\left[ 0, \| u\| _{L^{\infty }(\Om )}\right] } \widetilde{f},\\ 
M_{B}&:=& \widetilde{f_{B}}(\|u\|_{ L^{\infty }(\Om )})=\max _{\left[0, \| u\| _{L^{ \infty }(\Om )}\right] }\widetilde{f_{B}},\nonumber 
\end{eqnarray}
\color{black}
Throughout the proof, we use  that for any $\alpha >0$, there exist two constants $C_i=C_i(\al)>0$, $i=1,2$, independent of $x$, such that
\begin{equation*}% \label{Des4}
C_1(1+x^{\alpha })\le (1+x)^{\alpha }\le C_2(1+x^{\alpha }), \qquad \text{for all} \quad x\geq 0.
\end{equation*}

Using the same ideas as in the proof of \cite[Theorem 2.2]{Antonio_Arciga_Pardo_Sanchez},  by definition of $\widehat{f}$ (see  \eqref{f:hat}),  of $M$  (see \eqref{def:M:MB}), and using the H\"{o}lder inequality, with $s=r/q\in (1,\infty)$,  for any $t<q$, so that   $ts'=\frac{2^*}{2^*_{N/r}-1}$, we have that
\begin{align}\label{f:tilde:L2*}
&\int_{\Om} |\widetilde{f}(|u|)|^{\frac{2^*}{2^*_{N/r}-1}}dx \leq 
C\left(1 +\| u\|^{2^*}_{L^{2^*}(\Om )} \right),
\end{align}
and consequently,
\begin{equation}\label{estim:g:tilde}
\|\widehat{f}(\cdot,u)\|_{L^{q}(\Om)}
\le C \|a_2\|_{L^r(\Om)}\, M^{1-t/q}\,\left(1 +\| u\|^{2^*\left( \frac{1}{q}-\frac{1}{r}\right)}_{H^{1}(\Om )} \right) \,  ,
\end{equation}
where 
\begin{align}\label{Def:t}
t:=\frac{2^{*}}{2^{*}_{N/r}-1}\left( 1-\frac{q}{r}\right) <q \iff q >\frac{2N}{N+2} \ \checkmark\,.
\end{align}
On the other hand, in an analogous way, using H\"older inequality  with $s_B:=r_B/q_B$, and $t_B<q_B$ such that $t_Bs_B^{'}=\frac{2_{*}}{2_{*,N/r_B}-1}$, we obtain
\begin{align}\label{fB:tilde:L2*}
&\int_{\p \Om } |\widetilde{f}_{B}(|u|)| ^{\frac{2_*}{2_{*,N/r_B}-1}}dS  \le C\left(1 +\| u\|^{2_*}_{L^{2_*}(\p\Om )} \right),
\end{align}
consequently,
\begin{align}\label{fB:bound}
&\| f_B(\cdot,u)\|_{L^{q_B}(\p\Om )} \leq C \|a_B\|_{L^{r_B}(\p \Om )} 
M_{B}^{1-\frac{t_B}{q_B}}
\left( 1+\| u\|^{2_*\left( \frac{1}{q_B}-\frac{1}{r_B}\right) }_{H^{1}(\Om )} \right)  
\end{align}
where
\begin{align}\label{Def:tB}
t_B :=\frac{2_{*}}{2_{*,N/r_B}-1}\left( 1-\frac{q_B}{r_B}\right) <q_B \iff q_B >\frac{2(N-1)}{N}\,\checkmark\,.
\end{align}
From  \eqref{step1}, using  first 
\eqref{f:tilde:L2*}, \eqref{fB:tilde:L2*},
and next \eqref{estim:g:tilde}, \eqref{fB:bound}  we get 
\begin{align}\label{GN1a}
\| u\|_{L^\infty (\Om ) }
& \leq C\left[  \| \na u\|_{L^{2}(\Om)}^{\gamma} +  \| a\|_{L^{r}(\Om )} \,M^{1-\frac{t}{q}} \left( 1 +\| u\|_{L^{2^*}(\Om )} ^{2^*\left( \frac{1}{q}-\frac{1}{r}\right)}\right)\right.  \nonumber\\
&\quad\left. +  \| a_B\|_{L^{r_B}(\p \Om )}\, M_{B}^{1-\frac{t_B}{q_B}}
\left( 1+\| u\|^{2_*\left( \frac{1}{q_B}-\frac{1}{r_B}\right) }_{L^{2_*}(\p\Om )} \right) \right] ^{\sigma } \| u\|_{L^{2^*}(\Om )}^{(1-\sigma) } \\
\label{GN1b}
&\leq  C\left[ \|u\|_{H^{1}(\Om)}^{\gamma \s} +    \| a\|^{\sigma }_{L^{r}(\Om )}\,
M^{\left( 1-\frac{t}{q}\right)\sigma }\left( 1 +\| u\|_{H^{1}(\Om )} ^{2^*\left( \frac{1}{q}-\frac{1}{r}\right)\sigma }\right) \right. \nonumber \\
&\quad\left.  +   \| a_B\|^{\sigma }_{L^{r_B}(\p \Om )}\,M_{B}^{\left( 1-\frac{t_B}{q_B}\right) \sigma }  \left( 1+\| u\|^{2_*\left( \frac{1}{q_B}-\frac{1}{r_B}\right)\sigma }_{H^{1}(\Om )} \right) \right] \| u\|_{H^{1}(\Om )}^{(1-\sigma) }.
\end{align}
Taking into account the definitions of $M$ and $M_B$, see \eqref{def:M:MB}, that $\widetilde{f}$ and $\widetilde{f}_B$ are non-decreasing, and the definitions of the functions $\widetilde{h}$ and $\widetilde{h}_B$, see \eqref{def:h:hB}, we can write the following 
\begin{equation}\label{M:h,MB:hB}
M=\frac{\| u\|_{L^\infty (\Om ) }^{2^*_{N/r}-1}}{\widetilde{h}\big(\| u\|_{L^\infty (\Om ) }\big)} \qquad \text{and} \qquad M_{B}=\frac{\| u\|_{L^\infty (\Om ) }^{2_{*,N/r_B}-1}}{\widetilde{h}_{B}\big(\| u\|_{L^\infty (\Om ) }\big)}.
\end{equation}
Using the definitions of $s$,  and  $2^*_{N/p}$ (see  \eqref{Def:t} and \eqref{Def:Np:NpB} resp.), we  get  
\begin{equation*} 
1-\frac{t}{q}
=1-\frac{2^*}{2_{N/r}^*-1}\left( \frac{1}{q}-\frac{1}{r}\right) 
= \frac{2^*_{N/q}-1}{2^*_{N/r}-1}
\end{equation*}
and since the definition of $\s$,  (see \eqref{Def:sigma:2})
\begin{equation}\label{TRM3}
\left( 2^*_{N/r}-1\right) \big( 1-\tfrac{t}{q}\big)\sigma =1.
\end{equation}
Similarly, from the definitions of $s_B$,  and  $2_{*,N/q_B}$ (see
\eqref{Def:tB} and \eqref{Def:Np:NpB} resp.), we obtain
\begin{equation*} 
1-\frac{t_B}{q_B}=1-\frac{2_*}{2_{*,N/r_B}-1}\left( \frac{1}{q_B}-\frac{1}{r_B}\right)= \frac{2_{*,N/q_B}-1}{2_{*,N/r_B}-1}.   
\end{equation*}
Likewise, since the definition of $\s$, see \eqref{Def:sigma:2}, and the equivalences  \eqref{equiv:q:qB},  

\begin{equation}\label{TRM4}
\left( 2_{*,N/r_B}-1\right) \left( 1-\frac{t_B}{q_B}\right)\sigma =\frac{2_{*,N/q_B}-1}{2_{N/q}^*-1}=1.
\end{equation}
Now,   dividing both sides of the inequality \eqref{GN1b} by $\|u\|_{L^\infty (\Om ) }$, using the definitions of  $M$ and $M_B$,  \eqref{TRM3}, \eqref{TRM4}, and   $a_M$ defined in \eqref{Def:aM},

\begin{align*}
1\leq  \frac{ \|u\|_{H^{1}(\Om )}^{\gamma \s + (1-\sigma) }}{\| u\|_{L^\infty (\Om ) }} &+ Ca_{M}^{\sigma } \left[ \frac{ \left( 1+\|u\|_{H^{1}(\Omega  )}^{2^*\left( \frac{1}{q}-\frac{1}{r} \right) \sigma }\right) }{\widetilde{h}^{ \frac{1}{2^*_{N/r}-1}  }\big(\| u\|_{L^\infty (\Om ) }\big)}\right.\nonumber\\
&\qquad\qquad +    \left. \frac{\left( 1+\| u\|_{H^{1}(\Omega )}^{2_*\left( \frac{1}{q_B}-\frac{1}{r_B}\right)\sigma }\right)}{\widetilde{h}_{B}^{\,\frac{1}{ 2_{*,N/r_B}-1}}\big(\| u\|_{L^\infty (\Om ) }\big)}       \right]\| u\|_{H^{1}(\Om )}^{(1-\sigma) } .
\end{align*}
Now, the definition of $h$ (see \eqref{def:h}), implies that
\begin{equation*}
\frac{1}{h\big(\| u\|_{L^\infty (\Om ) }\big)^{\frac{1}{2^*_{N/r}-1} }}
=\max \left\{ \frac{1}{\widetilde{h}\big(\| u\|_{L^\infty (\Om ) }\big)^{\frac{1}{2^*_{N/r}-1} }}, 
\frac{1}{\widetilde{h}_{B}^{\,\frac{ 1 }{ 2_{*,N/r_B}-1}}\big(\| u\|_{L^\infty (\Om ) }\big)} 
\right\} .
\end{equation*}
And substituting this maximum 
we get
\begin{align*}
& 
h^{\frac{1}{2^*_{N/r}-1} }\big(\| u\|_{L^\infty (\Om ) }\big)
\leq \frac{h^{\frac{1}{2^*_{N/r}-1} }\big(\| u\|_{L^\infty (\Om ) }\big)
}{\| u\|_{L^\infty (\Om ) }}\,\|u\|_{H^{1}(\Om )}^{\gamma \s + (1-\sigma) } \\ 
& \qquad + Ca_{M}^{\sigma } \left(  1+\| u\|_{H^{1}(\Om )}^{2^*\left( \frac{1}{q}-\frac{1}{r} \right) \sigma } \right.  \left. +\| u\|_{H^{1}(\Om )}^{2_*\left( \frac{1}{q_B}-\frac{1}{r_B}\right)\sigma }\right) \| u\|_{H^{1}(\Om )}^{(1-\sigma) }. \nonumber 
\end{align*}
At this moment, thanks to  hypothesis \ref{ffB}, we observe that either 
\begin{equation*}%\label{h:s:bd}
\frac{\widetilde{h}(s)^{\, \frac{1}{2^*_{N/r}-1}}}{s} %=\frac{s}{s\,\widetilde{f}(s)^{\,\frac{1}{2^*_{N/r}-1}}}
=\frac{1}{\widetilde{f}(s)^\frac{1}{2^*_{N/r}-1}}
\to 0\q{as} s\to\infty, \qquad\text{or}   
\end{equation*}
\begin{equation*}
\frac{\widetilde{h}_B(s)^{\, \frac{1}{2^*_{N/r_B}-1}}}{s} =\frac{1}{\widetilde{f}_B(s)^\frac{1}{2^*_{N/r_B}-1}}
\to 0\q{as} s\to\infty,   
\end{equation*}
and so $h(s)^{\, \frac{1}{2^*_{N/r}-1}}/s\to 0$ as $\to\infty.$  Consequently, we can write

\begin{align} \label{Des:h}
h^{\frac{1}{2^*_{N/r}-1} }\big(\| u\|_{L^\infty (\Om ) }\big)
&\leq C \Bigg[ 1+ \|u\|_{H^{1}(\Om )}^{\gamma \s}+\| u\|_{H^{1}(\Om )}^{2^*\left( \frac{1}{q}-\frac{1}{r} \right) \sigma }  \\
& \qquad \qquad   +\| u\|_{H^{1}(\Om )}^{2_*\left( \frac{1}{q_B}-\frac{1}{r_B}\right)\sigma }\Bigg] \| u\|_{H^{1}(\Om )}^{(1-\sigma) } \nonumber 
\end{align}
for $C=C(N, r, r_B,|\Om |,|\p \Om |, a_M )>0$, a constant independent of $u$.
\medskip

Now we look for the largest exponent of the three terms. 
Let $E_M$ denote the following related maximum 
\begin{equation}\label{def:EM}
E_M:=\max \left\{\gamma , 2^*\left( \frac{1}{q}-\frac{1}{r} \right) 
, \frac{2^*}{q}-\frac{2^*}{N}-\frac{2_*}{r_B}\right\},
\end{equation}
where we have used \eqref{equiv:q:qB}. Raising both sides of \eqref{Des:h} to the power $(2^*_{N/r}-1)$, and using \eqref{def:EM},
\begin{equation*}%\label{Des:h:2}
h \bigl(\|u\|_{L^\infty(\Omega)}\bigr) \le C\Big(1+\|u\|_{H^1(\Omega)}^{\beta}  \Big),
\end{equation*}
where, 
\begin{equation*}%\label{Def:Beta:A}
\beta
:= \Big(E_M+\frac{1-\sigma}{\s}\Big) \overline{\te},\q{with}      
\overline{\te} := (2^*_{N/r}-1)\s .
\end{equation*}
Since \eqref{Def:sigma:2}, and introducing the definition of  $E_M$ (see \eqref{def:EM})
\begin{equation}\label{Def:Beta:C} 
\be:=\max \left\{
\gamma+2^*_{N/q}-2 , 2^*_{N/r}-2
, \tfrac{2^*}{N}-\tfrac{2_*}{r_B}\right\}\times \overline{\te},\ \overline{\te}  = \frac{2^*_{N/r}-1}{2^*_{N/q}-1},
\end{equation}
where $\g=\g(N,l, q)$ is given by \eqref{Def:gamma:2}.

Finally, $\be_0$ is the infimum of the exponent $\be(q)$
$$
\be_0:=\inf \be(q), \qq{for}q\in\left(\tfrac{N}{2},\min\left\{r,\widetilde{r_B}\right\}\right).
$$
Observe that $\be_0=\be_0(N,l,r,r_B)$ and it is independent of $u$, ending the proof of Theorem \ref{th:RegEst}.
\end{proof}

\begin{proof}[Proof of Corollary \ref{cor1}]
Using  \eqref{GN1a}, 
reasoning as in \eqref{M:h,MB:hB}-\eqref{Des:h}, and raising both sides of \eqref{Des:h:2*} to the power $(2^*_{N/r}-1)$, we deduce
\begin{align} \label{Des:h:2*}
h\big(\| u\|_{L^\infty (\Om ) }\big)
&\leq C \Bigg[ 1+ \| \na u\|_{L^{2}(\Om)}^{\gamma}+\| u\|_{L^{2^*}(\Om )}^{2^*\left( \frac{1}{q}-\frac{1}{r} \right)}  \nonumber\\
& \qquad \qquad   +\| u\|_{L^{2_*}(\p\Om )}^{2_*\left( \frac{1}{q_B}-\frac{1}{r_B}\right) }\Bigg]^{\s(2^*_{N/r}-1)} \| u\|_{L^{2^*}(\Om )}^{(1-\sigma)(2^*_{N/r}-1) }  
\end{align}
for $C=C(N, r, r_B,|\Om |,|\p \Om |, a_M )>0$, a constant independent of $u$.

Now, using \eqref{eq:H1:bd}, \eqref{f:tilde:L2*}, \eqref{fB:tilde:L2*} ,
and regrouping terms, we deduce 
\begin{align}
\|u\|_{H^{1}(\Omega)}^2 \label{eq:H1:bd:pde:2}
&\le Ca_M\Big(1 + \|u\|_{L^{2^*}(\Omega)}^{\frac{2 }{2-l}}+ \| u\|^{2^*_{N/r}}_{L^{2^*}(\Om )}  
+\| u\|^{2_{*,N/ r_{B}}}_{L^{2_*}(\p\Om )} \Big).
\end{align}

Introducing now \eqref{eq:H1:bd:pde:2} into \eqref{Des:h:2*}, we can write
\begin{align*}%\label{Des:h:3}
&h\big(\|u\|_{L^\infty(\Omega)}\big)\le
C\Bigg[1 +\Big(1+\|u\|_{L^{2^*}(\Omega)}^{\frac{2}{2-l}}
+\|u\|_{L^{2^*}(\Omega)}^{2^*_{N/r}}+\|u\|_{L^{2_*}(\partial\Omega)}^{\,2_{*,N/r_B}}\Big)^{\frac{\gamma}{2}} \nonumber\\
&\qquad\qquad+\|u\|_{L^{2^*}(\Omega)}^{2^*\left(\frac1q-\frac1r\right)}
+\|u\|_{L^{2_*}(\partial\Omega)}^{2_*\left(\frac1{q_B}-\frac1{r_B}\right)}\Bigg]^{\sigma(2^*_{N/r}-1)}\times\|u\|_{L^{2^*}(\Omega)}^{(1-\sigma)(2^*_{N/r}-1)}.
\end{align*}
Defining
\begin{align*}%\label{Def:Beta:123:q} 
\be_1(q)&:=\max \left\{ \frac{2\g \s (2^*_{N/r}-1)}{(2-l)}, 2^*_{N/r}\tfrac{\g (2^*_{N/r}-1)}{2}, 2^*\left(\tfrac1{q}-\tfrac1{r}\right)\s (2^*_{N/r}-1)\right\},\\
\be_2(q)&:=\max \left\{ 2^*_{N/r_B}\tfrac{\g}{2} \s (2^*_{N/r}-1), 2_*\left(\tfrac1{q_B}-\tfrac1{r_B}\right)\s (2^*_{N/r}-1)\right\},\\
\be_3(q)&:= (1-\sigma)(2^*_{N/r}-1),
\end{align*}
and
\begin{align*}%\label{Def:Beta:123} 
\be_i:=\inf_{q\in(\tfrac{N}{2},\min\left\{r,\widetilde{r_B} \right\})}\ \be_i(q) ,\qq{for}i=1,2,3,
\end{align*}
where $\be_i=\be_i(N,l, r, r_B)>0$,  are independent of $u$,
we end the proof of Corollary \ref{cor1}.
\end{proof}

\section{Explicit definition of \texorpdfstring{$\beta_0$}{}}
\label{sec:expl}
In this section we explicitly calculate  the value of the exponent $\be_0$ in  Theorem \ref{th:RegEst}. The calculations are lengthy, but we include them here by the sake of completeness.

\begin{cor}
\label{coro}
Assume that all the hypotheses of Theorem \ref{th:RegEst} hold. Then,  the exponent $\be_0$ in  \eqref{estim} take the values  of the tables \ref{Table1}.
\end{cor}  
\begin{proof}
The proof consists of two parts.
Part I  analyzes the maximum of the elements  defining $E_M$, see \eqref{def:EM}, in terms of sub-intervals of $l$, and consequently defines  explicitly $\be(q)$ for the different sub-intervals.
Part II explicitly defines $\be_0$, the infimum of $\be(q)$ in the different subintervals,  analyzing carefully the behavior of $\be(q)$.

\bigskip

{\bf Part I.} {\it Explicit definition   of $\be(q)$, 
on each sub-interval.}

Note that
\begin{equation}\label{max}
2^*_{N/r}-2
> \tfrac{2^*}{N}-\tfrac{2_*}{r_B}\iff \tfrac{1}{N}-\tfrac{1}{r}> -\tfrac{N-1}{Nr_B}
\iff r>\widetilde{r_B}.
\end{equation}

The definition of $\be=\be(q)$, for $q\in\left(\tfrac{N}{2},\min\left\{r,\widetilde{r_B}\right\}\right)$, is given as follows (see \eqref{Def:Beta:C}):
\begin{enumerate}[label=\textbf{(\arabic*)}$_q$, leftmargin=.2cm]
\setlength{\parindent}{0em}
\item
Assume on the one hand 
\begin{equation}\label{r:<:rB}
r\ge \widetilde{r_B}\iff l_1:=\tfrac{2_*}{r_B}+2^*\big(\tfrac{1}{N}-\tfrac{1}{r}\big)\ge 0,    
\end{equation} 
see \eqref{def:rB:tilde}. Then $q\in\left(\tfrac{N}{2},\widetilde{r_B}\right) $,
and $\left(\tfrac{N}{2},\widetilde{r_B}\right)\ne\emptyset $ for any $r_B>N-1.$
\begin{enumerate}[label=\textbf{(1.\alph*)}$_q$, leftmargin=.2cm]
\setlength{\parindent}{0em}
\item
\label{1a:q}
If now we assume also 
\begin{equation}\label{q:1a}
lq\le 2, \qq{then} \boxed{\g=l},\qquad \text{for}
\end{equation} 
$q\in\left(\tfrac{N}{2},\widetilde{r_B}\right)\cap
\left(\tfrac{N}{2},\tfrac{2}{l}\right]\ne\emptyset $ for any $\boxed{l<4/N}$.

\begin{enumerate}[label=\textbf{(1.a.\alph*)$_q$},leftmargin=.2cm]
\setlength{\parindent}{0em}
\item
\label{1aa:q}
Assume either  
\begin{equation*}%\label{q:1aa}
l+2^*_{N/q} > 2^*_{N/r}
\iff q>1 /\left(\tfrac{l}{2^*}+\tfrac{1}{r}\right),
\end{equation*} 
and using \eqref{max},  and \eqref{Def:Beta:C} 
\begin{equation}\label{be:1aa}
\be(q)=(l+2^*_{N/q}-2)\, \overline{\te}=(l-1)\, \overline{\te}+ 2^*_{N/r}-1,\q{for}
\end{equation}
\begin{equation}\label{I1aa}
    q\in I_{1,a,a}:=\left(\tfrac{N}{2},\widetilde{r_B}\right)\cap
\left(\tfrac{N}{2},\tfrac{2}{l}\right]\cap \left(1 /\left(\tfrac{l}{2^*}+\tfrac{1}{r}\right),\widetilde{r_B}\right).
\end{equation}
Observe that  the third sub-interval is non empty  if and only if
\begin{equation*}%\label{l:1aa}
\tfrac{l}{2^*}+\tfrac{1}{r}>
\tfrac{1}{N}+\tfrac{N-1}{Nr_B}
\iff l>l_1,
\end{equation*}
and so  
\begin{equation}\label{I:1aa:nonempty}
\qq{if}\boxed{l_1<l<4/N}\implies I_{1,a,a} \ne\emptyset .
\end{equation}
Since \eqref{r:<:rB}, the lower bound of $l$ in \eqref{I:1aa:nonempty}
is non-negative.
\item
\label{1ab:q} Assume either
\begin{equation*}%\label{q:1ab}
l+2^*_{N/q} \le 2^*_{N/r}
\iff q\le 1 /\left(\tfrac{l}{2^*}+\tfrac{1}{r}\right),
\end{equation*}
then   using \eqref{max},   and \eqref{Def:Beta:C} 
\begin{equation}\label{be:1ab}
\be(q)=(2^*_{N/r}-2) \overline{\te},\q{for}
\end{equation}
\begin{equation}\label{I1ab}
 q\in I_{1,a,b}:=\left(\tfrac{N}{2},\widetilde{r_B}\right)\cap
\left(\tfrac{N}{2},\tfrac{2}{l}\right]\cap \left(\tfrac{N}{2},1 /\left(\tfrac{l}{2^*}+\tfrac{1}{r}\right)\right].
\end{equation}
The third sub-interval is non empty  if and only if
\begin{equation}\label{l:1ab}
\tfrac{l}{2^*}+\tfrac{1}{r}<\tfrac{2}{N}
\iff l<2^*\big(\tfrac{2}{N}-\tfrac{1}{r}\big)=:l_2,
\end{equation}
and so  
\begin{equation*}%\label{I:1ab:nonempty}
\qq{if}\boxed{l<\min\big\{4/N,l_2\big\}}\implies I_{1,a,b} \ne\emptyset ,\checkmark
\end{equation*}
moreover
$$
\min\big\{4/N,l_2\big\}=
\begin{cases}
4/N &\text{ if }    r>N^2/4,\\
l_2&\text{ if }    r\le N^2/4.
\end{cases}
$$
\end{enumerate}

\item
\label{1b:q}
If now we assume on the contrary that 
\begin{equation}\label{ga:1b:q}
lq> 2, \q{then} \boxed{\g=\frac{2\,[\frac{N}{q}(1-l)+l]}{N(1-l)+2}},\text{ for }
q\in \left(\tfrac{N}{2},\widetilde{r_B}\right)\cap \left(\tfrac{2}{l},\widetilde{r_B}\right)
\end{equation} 
and
\begin{equation}\label{l3}
\left(\tfrac{2}{l},\widetilde{r_B}\right)\ne\emptyset \q{for any} \boxed{l>2\big[\tfrac{1}{N}+\tfrac{N-1}{Nr_B}\big]=:l_3}.    
\end{equation}

Moreover, using the definition of $\g$ in \eqref{ga:1b:q}, we have
\begin{align}
(\gamma+2^*_{N/q}-2) \, \overline{\te}
&=\left[\frac{(\g-1)}{2^*_{N/q}-1}+ 1\right](2^*_{N/r}-1) 
\label{ga:1b:2}
=\frac{2(2-l)}{[N(1-l)+2]}(2^*_{N/r}-1),
\end{align}
which is independent of $q$.
\begin{enumerate}[label=\textbf{(1.b.\alph*)}$_q$, leftmargin=.2cm]
\item
\label{1ba:q}
Assume either
\begin{equation*}%\label{q:1ba}
\gamma+2^*_{N/q} > 2^*_{N/r}
\iff q>\frac1{\tfrac{N+2}{2N}-\big(\tfrac{2}{N}-\tfrac{1}{r}\big)\tfrac{N(1-l)+2}{2(2-l)}}=\tfrac{2Nr(2-l)}{(N-2)rl+N[N(1-l)+2]} =:r_{1B},
\end{equation*} 
then   using \eqref{max},   \eqref{Def:Beta:C}  and \eqref{ga:1b:2}
\begin{equation}\label{be:1ba}
\be(q)=\frac{2(2-l)}{[N(1-l)+2]}(2^*_{N/r}-1) \q{is independent of} q,
\text{ for}
\end{equation}
\begin{equation}\label{I1ba}
q\in I_{1,b,a}:=\left(\tfrac{N}{2},\widetilde{r_B}\right)\cap
\left(\tfrac{2}{l},\widetilde{r_B}\right)\cap \left(\tfrac{2Nr(2-l)}{(N-2)rl+N[N(1-l)+2]},\widetilde{r_B}\right).
\end{equation}
The third sub-interval is nonempty whenever
\begin{align}\label{l:1ba}
&\tfrac{1}{N}+\tfrac{N-1}{Nr_B}<\tfrac{1}{2-l}\left[\tfrac{(N-2)l}{2N}+\tfrac{N(1-l)+2}{2r} \right]\nonumber\\
&\iff 2\big[\tfrac{1}{N}+\tfrac{N-1}{Nr_B}\big]-\tfrac{N+2}{2r}<
l\big[\tfrac{N-1}{Nr_B}+\tfrac{1}{2}-\tfrac{N}{2r}\big].
\end{align}

Assume either  
\begin{equation*}%\label{I:1ba:nonempty}
\left.
\begin{aligned}
&
r>\frac{\tfrac{N}{2}}{\tfrac{N-1}{Nr_B}+\tfrac{1}{2}},\text{ and }
l>\frac{l_3-\tfrac{N+2}{2r}}{\tfrac{N-1}{Nr_B}+\tfrac{1}{2}-\tfrac{N}{2r}},\\
&\text{or  }r=\frac{\tfrac{N}{2}}{\tfrac{N-1}{Nr_B}+\tfrac{1}{2}},
\qq{for all} l,\\
&\text{or  }r<\frac{\tfrac{N}{2}}{\tfrac{N-1}{Nr_B}+\tfrac{1}{2}},\
l<\frac{\tfrac{N+2}{2r}-l_3}{\tfrac{N}{2r}-\tfrac{N-1}{Nr_B}-\tfrac{1}{2}}\\
&\text{or  }r=\frac{N+2}{4\big[\tfrac{N-1}{Nr_B}+\tfrac{1}{N}\big]}>\frac{N(N-1)}{N-2},
\end{aligned}
\right\}
\implies
I_{1,b,a} \ne\emptyset .
\end{equation*}
Observe that 
$$
\frac{\tfrac{N}{2}}{\tfrac{N-1}{Nr_B}+\tfrac{1}{2}}<\frac{N+2}{4\big[\tfrac{N-1}{Nr_B}+\tfrac{1}{N}\big]},
$$
and so, when $r<
\frac{\tfrac{N}{2}}{\tfrac{N-1}{Nr_B}+\tfrac{1}{2}}$,
the LHS and the RHS of \eqref{l:1ba} are both negative.
\item
\label{1bb:q} Assume either 
\begin{equation*}%\label{q:1bb}
\gamma+2^*_{N/q} \le 2^*_{N/r}
\iff q\le \tfrac{2Nr(2-l)}{(N-2)rl+N[N(1-l)+2]},
\end{equation*} 
then   using \eqref{max},  and \eqref{Def:Beta:C} 
\begin{equation}\label{be:1bb}
\be(q)=(2^*_{N/r}-2) \frac{2^*_{N/r}-1}{2^*_{N/q}-1},\q{for}
\end{equation} 
\begin{equation}\label{I1bb}
    q\in I_{1,b,b}:=\left(\tfrac{N}{2},\widetilde{r_B}\right)\cap
\left(\tfrac{2}{l},\widetilde{r_B}\right)\cap \left(\tfrac{N}{2},\tfrac{2Nr(2-l)}{(N-2)rl+N[N(1-l)+2]}\right]. 
\end{equation}
The third sub-interval is nonempty whenever
\begin{align*}%\label{l:1ba}
\tfrac{1}{2}<\tfrac{2r(2-l)}{(N-2)rl+N[N(1-l)+2]}
&\iff [(N-2)r-N^2+4r]l<8r-2N-N^2\\
&\iff [(N+2)r-N^2]l<8r-N(N+2).
\end{align*}
\begin{equation*}%\label{I:1bb:nonempty}
\left.\begin{cases}
\text{If either  }r>\tfrac{N^2}{N+2},&\q{and}
l<\tfrac{8r-N(N+2)}{(N+2)r-N^2},\\
\text{either  }r<\tfrac{N^2}{N+2},
&\q{and} l>\tfrac{N(N+2)-8r}{N^2-(N+2)r},
\end{cases}\right\} \implies I_{1,b,b} \ne\emptyset .
\end{equation*}
\end{enumerate}
\end{enumerate}
\bigskip

\item
Assume  on the other hand  
\begin{equation}\label{r:rtildeB}
r< \widetilde{r_B} \iff  2^*\big(\tfrac{1}{r}-\tfrac{1}{N}\big)- \tfrac{2_*}{r_B}> 0,
\end{equation} 
see \eqref{def:rB:tilde}. Then $q\in\left(\tfrac{N}{2},r\right)$ with $r<\widetilde{r_B} .$ 

\begin{enumerate}[label=\textbf{(2.\alph*)}$_q$, leftmargin=.2cm]
\item
%\label{2a:q}
Now we assume also \eqref{q:1a} , but for
$q\in\left(\tfrac{N}{2},r\right)\cap
\left(\tfrac{N}{2},\tfrac{2}{l}\right]$, and as in \ref{1a:q}, this interval is non empty for any $l<4/N$.

\begin{enumerate}[label=\textbf{(2.a.\alph*)}$_q$, leftmargin=.2cm]
\item
\label{2aa:q}
Either  
$$
l+2^*_{N/q}-2 > \tfrac{2^*}{N}-\tfrac{2_*}{r_B}\iff q>1 \big/\big(\tfrac{l}{2^*}+\tfrac{1}{\widetilde{r_B}}\big),
$$  
and then   using \eqref{max},  and \eqref{Def:Beta:C}, $\be(q)$  is as given by \eqref{be:1aa}, but now for 
\begin{equation}\label{I2aa}
    q\in I_{2,a,a}:=\left(\tfrac{N}{2},r\right)\cap
\left(\tfrac{N}{2},\tfrac{2}{l}\right]\cap \big(1 /\big(\tfrac{l}{2^*}+\tfrac{1}{\widetilde{r_B}}\big),r\big).
\end{equation}
The third sub-interval is non empty  if and only if
\begin{equation*}%\label{l:2aa}
\tfrac{l}{2^*}+\tfrac{1}{N}+\tfrac{N-1}{Nr_B}>
\tfrac{1}{r}
\iff l>2^*\big(\tfrac{1}{r}-\tfrac{1}{N}\big)-\tfrac{2_*}{r_B},
\end{equation*}
and so  
\begin{equation}\label{I:2aa:nonempty}
\qq{if}\boxed{2^*\big(\tfrac{1}{r}-\tfrac{1}{N}\big)-\tfrac{2_*}{r_B}<l<4/N}\implies I_{2,a,a} \ne\emptyset .\checkmark
\end{equation}

\item
%\label{2ab:q} 
Either 
$$
l+2^*_{N/q}-2 \le \tfrac{2^*}{N}-\tfrac{2_*}{r_B} \iff q\le 1 \big/\big(\tfrac{l}{2^*}+\tfrac{1}{\widetilde{r_B}}\big),
$$ 
and then   using \eqref{max},  and \eqref{Def:Beta:C}, 
$$\be(q):=\left(\frac{2^*}{N}-\frac{2_*}{r_B}\right)\frac{2^*_{N/r}-1}{2^*_{N/q}-1},
$$
now for 
\begin{equation}\label{I2ab}
    q\in I_{2,a,b}:=\left(\tfrac{N}{2},r\right)\cap
\left(\tfrac{N}{2},\tfrac{2}{l}\right]\cap \big(\tfrac{N}{2},1 /\big(\tfrac{l}{2^*}+\tfrac{1}{\widetilde{r_B}}\big)\big].
\end{equation}
The third sub-interval is non empty  if and only if
\begin{equation*}%\label{l:2ab}
\tfrac{l}{2^*}+\tfrac{1}{N}+\tfrac{N-1}{Nr_B}<
\tfrac{2}{N}
\iff l<\tfrac{2}{N-2}-\tfrac{2_*}{r_B},
\end{equation*}
and so 
\begin{equation*}%\label{I:2ab:nonempty}
\qq{if}\boxed{l<\min\big\{4/N,\tfrac{2}{N-2}-\tfrac{2_*}{r_B}\big\}}
\implies I_{2,a,b} \ne\emptyset .\checkmark
\end{equation*}
\end{enumerate}

\item
%\label{2b:q}
Now, assuming \eqref{ga:1b:q}, 
reasoning as in \ref{1b:q},
\eqref{ga:1b:2} holds, but now  for $q\in \left(\tfrac{N}{2},r\right)\cap
\left(\tfrac{2}{l},r\right)\ne\emptyset$ for any $\boxed{l>2/r}$.\\

\begin{enumerate}[label=\textbf{(2.b.\alph*)}$_q$, leftmargin=.2cm]
\item
Either 
\begin{equation}\label{ineq:q:2ba}
\gamma+2^*_{N/q}-2 > \tfrac{2^*}{N}-\tfrac{2_*}{r_B}\iff q>\tfrac{2Nr_B(2-l)}{r_B[N+2(1-l)]+(N-1)[N(1-l)+2]},
\end{equation}
then   using \eqref{max},  and \eqref{Def:Beta:C},  
$\be(q)$ is as given by \eqref{be:1ba}, but now for 
\begin{equation}\label{I2ba}
    q\in I_{2,b,a}:=\left(\tfrac{N}{2},r\right)
\cap\left(\tfrac{2}{l},r\right)
\cap\left(\tfrac{2Nr_B(2-l)}{r_B[N+2(1-l)]+(N-1)[N(1-l)+2]},r\right).
\end{equation}
The third sub-interval is  nonempty  if and only if
\begin{align*}
& \tfrac{2Nr_B(2-l)}{r_B[N+2(1-l)]+(N-1)[N(1-l)+2]}<r \\
&\iff 2Nr_B + (1-l)\big(2Nr_B-2rr_B-N(N-1)\big)<r\big[r_BN +2(N-1)\big].
\end{align*}
Observe that
\begin{align}
&2Nr_B-2rr_B-N(N-1)>0\iff \frac{N[2r_B-(N-1)]}{2r_B}>r\nonumber\\
\label{r:rB}
&\iff \frac{1}{r}>\frac{2r_B\mp (N-1)}{N[2r_B-(N-1)]}=\frac{1}{N}
+\frac{N-1}{N[2r_B-(N-1)]}.     
\end{align}
Using \eqref{r:rtildeB} and that $\tfrac{1}{r_B}>\frac{1}{2r_B-(N-1)}\iff 2r_B-(N-1)>r_B\ \checkmark,$ we deduce that
\eqref{r:rB} holds,  hence the third sub-interval is  nonempty  if and only if
\begin{equation*}%\label{l:2ba}
\boxed{l>\max\Big\{2/r,1+\tfrac{2Nr_B-r[r_BN +2(N-1)]}{2Nr_B-2rr_B-N(N-1)}\Big\}.}
\end{equation*}
\item
%\label{2bb:q} 
Either 
\begin{equation}\label{ineq:q:2bb}
\gamma+2^*_{N/q}-2 \le \tfrac{2^*}{N}-\tfrac{2_*}{r_B}
\iff q\le \tfrac{2Nr_B(2-l)}{r_B[N+2(1-l)]+(N-1)[N(1-l)+2]},
\end{equation} 
then   using \eqref{max},  and \eqref{Def:Beta:C} 
\begin{equation}\label{be:2bb} 
\be(q)=\left(\frac{2^*}{N}-\frac{2_*}{r_B}\right) \overline{\te},\qquad \text{for}
\end{equation}
\begin{equation}\label{I2bb} 
q\in I_{2,b,b}:=\left(\tfrac{N}{2},r\right)
\cap\left(\tfrac{2}{l},r\right)
\cap \left(\tfrac{N}{2},\tfrac{2Nr_B(2-l)}{r_B[N+2(1-l)]+(N-1)[N(1-l)+2]}\right].
\end{equation}
The third sub-interval is  nonempty  if and only if
\begin{align*}
&\tfrac{1}{2} < \tfrac{2r_B(2-l)}{r_B[N+2(1-l)]+(N-1)[N(1-l)+2]} \\
&\iff [N(N-1)-2r_B](1-l)< -(N-4)r_B-2(N-1),
\end{align*}
which holds   if and only if
\begin{align*}%\label{l:2bb}
&\text{either }r_B< \tfrac{1}{2}\ N(N-1),\q{and }l>1+\tfrac{(N-4)r_B+2(N-1)}{N(N-1)-2r_B},\\
&\text{either } r_B> \tfrac{1}{2}\ N(N-1),\q{and }l<1-\tfrac{(N-4)r_B+2(N-1)}{2r_B-N(N-1)}. 
\end{align*}
At this point, observe that for $r_B> \tfrac{1}{2}\ N(N-1),$
\begin{equation*}
1<\tfrac{(N-4)r_B+2(N-1)}{2r_B-N(N-1)} \iff  0<(N-6)r_B+(N+2)(N-1).
\end{equation*}
This inequality holds for any $N\ge 6.$ If $N<6,$ the above inequality holds if and only if $r_B<\frac{(N+2)(N-1)}{6-N}.$ Consequently, $I_{2,b,b}\ne\emptyset$ if and only if $l>2/r$ and

\begin{align*}%\label{l:2bb}
&\text{either }\boxed{r_B< \tfrac{1}{2}\ N(N-1),\q{and  }l>1+\tfrac{(N-4)r_B+2(N-1)}{N(N-1)-2r_B}},\\
&\text{or } \boxed{N\ge 6,\ r_B> \tfrac{1}{2}\ N(N-1)\text{ and }l<1-\tfrac{(N-4)r_B+2(N-1)}{2r_B-N(N-1)}},
\\
&\text{or } \boxed{N< 6,\ r_B\in \big(\tfrac{N(N-1)}{2}\ ,\tfrac{(N+2)(N-1)}{6-N}\big)},\\
&\qquad\qquad \q{and }\boxed{ l<1-\tfrac{(N-4)r_B+2(N-1)}{2r_B-N(N-1)}}. 
\end{align*}
\end{enumerate}
\end{enumerate}
\end{enumerate}

\medskip
  
{\bf Part II.} {\it Explicit definition of $\be_0$ on each sub-interval.}
\medskip

Now, we look for the infimum of the exponent $\be(q)$
$$
\be_0=\be_0(i,\al,\be):=\inf \be(q) \q{for} q\in I_{i,\al,\be},\ i=1,2,\ \al,\be\in\{a,b\}.
$$
Note   that, 
\begin{equation}\label{theta-derivative}
\overline{\te}\,'(q) = - \frac{2^*(2^*_{N/r}-1)}{q^2\bigl(2^*_{N/q}-1\bigr)^2} 
= - \frac{2^*\ \overline{\te}}{q^2\bigl(2^*_{N/q}-1\bigr)}
<0.
\end{equation}

\begin{enumerate}[label=\textbf{(\arabic*)}, leftmargin=.2cm]
\setlength{\parindent}{0em}
\item
%\label{r<rB:0}
Assume  $r\ge \widetilde{r_B}.$
\begin{enumerate}[label=\textbf{(1.\alph*)}$_0$,leftmargin=.2cm]
\setlength{\parindent}{0em}
\item
%\label{1a:0} 
The interval for $q$ is now defined as $q\in \left(\tfrac{N}{2},\widetilde{r_B}\right)\cap \left(\tfrac{N}{2},\tfrac{2}{l}\right]$, 
see \ref{1a:q}.
Observe that the second sub-interval is nonempty for any $l<4/N,$ and observe that $4/N<1+2/N.$ 
\begin{enumerate}[label=\textbf{(1.a.\alph*)}$_0$,leftmargin=.2cm]
\setlength{\parindent}{0em}
\item 
\label{1aa:0} 
The function $\be=\be(q)$ is defined by \eqref{be:1aa}  for $q$ on the interval  $I_{1,a,a}$,  given by \eqref{I1aa}, see \ref{1aa:q}.
\medskip

If $l<1$ then $\be$ is increasing, so the infimum is reached at the maximum of the lower limits of the subintervals defining $I_{1,a,a}$ (see \eqref{I1aa}), either
$\be\big(1 /\left(\frac{l}{2^*}+\frac{1}{r}\right)\big)$
if $N/2< 1 /\left(\tfrac{l}{2^*}+\tfrac{1}{r}\right)$, or  $\be(N/2)$ otherwise. Observe that
\begin{equation*}%\label{N:l}
N/2<1 /\left(\tfrac{l}{2^*}+\tfrac{1}{r}\right)\iff l< l_2. 
\end{equation*}

If $l>1$ then $\be$ is decreasing, see \eqref{theta-derivative}, so the infimum is reached at the minimum of the upper limits of the subintervals defining  $I_{1,a,a}$, either $\be(\widetilde{r_B})$ if $\widetilde{r_B}<2/l$, or $\be(2/l)$ otherwise, and
\begin{equation}\label{rB:l}
\widetilde{r_B}\le 2/l\iff l\le l_3.    
\end{equation}
Besides, $1<l_3\iff r_B<2_*$, which is only possible when $N<4$,  where we have used that $r_B>N-1$.
Then, for $\boxed{l<\tfrac{4}{N}},$
$$
\be_0(1,a,a)=
\begin{cases}
\boxed{(l-1)\, \frac{(2^*_{N/r}-1)}{2^*_{N(\frac{l}{2^*}+\frac{1}{r})}-1} + 2^*_{N/r}-1} \\
\qquad\qquad\qquad\qquad\qquad\qquad
\text{if } l< \min\big\{2^*\left(\tfrac{2}{N}-\tfrac{1}{r}\right),1\big\},\\[.1cm]
\boxed{l(2^*_{N/r}-1)}\qquad\qquad\qquad\ \
\text{if } l\in \big[2^*\left(\tfrac{2}{N}-\tfrac{1}{r}\right),1\big],\\
\boxed{(l-1)\, \frac{(2^*_{N/r}-1)}{2_{*N/r_B}-1} + 2^*_{N/r}-1}\\
\qquad\qquad\qquad\qquad\qquad\qquad
\text{if } l\in \big(1,l_3\big],\\ 
\boxed{(l-1)\, \frac{(2^*_{N/r}-1)}{2^*_{Nl/2}-1} + 2^*_{N/r}-1}\\
\qquad\qquad\qquad\qquad\qquad\qquad
\text{if } l>\max\big\{1,l_3\big\}.
\end{cases}
$$ 
\medskip

\item
\label{1ab:0}  
The function $\be=\be(q)$ is defined by \eqref{be:1ab}  for $q$ on the interval   $I_{1,a,b}$,  given by \eqref{I1ab}, see \ref{1ab:q}.

Now $\be$ is decreasing (see \eqref{theta-derivative}), so the infimum is reached at the minimum of the upper limits of the subintervals defining  $I_{1,a,b}$, either $\be(\widetilde{r_B})$, or $\be(2/l)$, or $\be\big(1 /\left(\frac{l}{2^*}+\frac{1}{r}\right)\big)$. Observe that
\begin{align*}
%\label{r:til:B:l:2*}
\widetilde{r_B} \le  1 /\left(\tfrac{l}{2^*}+\tfrac{1}{r}\right) &\iff l\le l_1,\\
%\label{2:l:l:2*}
\tfrac{2}{l} \le  1 /\left(\tfrac{l}{2^*}+\tfrac{1}{r}\right) &\iff l\ge N/r. 
\end{align*}
Using the above and \eqref{rB:l}, 
$$
\be_0(1,a,b)=
\begin{cases}
\be(\widetilde{r_B}),\quad
\text{if } 
l\le \min\big\{l_1, 
\tfrac{2}{N}+\tfrac{2(N-1)}{Nr_B}\big\} ,\\ 
\be\big(1 /\left(\frac{l}{2^*}+\frac{1}{r}\right)\big) ,\quad
\text{if } l_1\leq l\leq \frac{N}{r} ,\\
\be(2/l),\quad
\text{if } 
l\ge \max\big\{N/r, l_3\big\}.
\end{cases}
$$
Observe that 
\begin{align*}
&l_1< l_3     \iff r<\tfrac{N^2/2}{\frac{N-1}{r_B}+1}\\
&\iff l_1< \tfrac{N}{r}\iff 2\left[\tfrac{1}{N}+\tfrac{N-1}{Nr_B}\right]<\tfrac{N}{r}.    
\end{align*}

So,   for $\boxed{l<\tfrac{4}{N}},$ $\boxed{r<\tfrac{N^2/2}{\tfrac{N-1}{r_B}+1}},$ 
we have
$$
\tfrac{2_*}{r_B} +2^*\!\left(\tfrac1N-\tfrac1r\right) < 2\!\left[\tfrac1N+\tfrac{N-1}{Nr_B}\right] < \tfrac{N}{r},\quad\text{and}
$$
$$
\be_0(1,a,b):=
\begin{cases}
\boxed{\frac{(2^*_{N/r}-2)\,(2^*_{N/r}-1)}{2_{*N/r_B}-1}}\quad
\text{if }  
l\le \tfrac{2_*}{r_B}+2^*\!\left(\tfrac1N-\tfrac1r\right) ,\\ 
\boxed{\frac{(2^*_{N/r}-2)\,(2^*_{N/r}-1)}{2^*_{N(\frac{l}{2^*}+\frac{1}{r})}-1}}\quad
\text{if } \tfrac{2_*}{r_B}
+2^*\!\left(\tfrac1N-\tfrac1r\right)
\le l\le \tfrac{N}{r},\\
\boxed{\frac{(2^*_{N/r}-2)\,(2^*_{N/r}-1)}{2^*_{Nl/2}-1}}\quad
\text{if } 
l\ge \tfrac{N}{r}.
\end{cases}
$$

\bigskip

\noindent
Now,for $\boxed{l<\tfrac{4}{N}},$  $\boxed{r\ge \tfrac{N^2/2}{\tfrac{N-1}{r_B}+1}}$,  we have 
\[
\tfrac{N}{r}\le 2\!\left[\tfrac1N+\tfrac{N-1}{Nr_B}\right]
\le
\tfrac{2_*}{r_B}
+2^*\!\left(\tfrac1N-\tfrac1r\right),\quad\text{and}
\]
$$
\be_0(1,a,b):=
\begin{cases}
\boxed{\frac{(2^*_{N/r}-2)\,(2^*_{N/r}-1)}{2_{*N/r_B}-1}}\quad
\text{if }  
l\le 2\!\left[\tfrac1N+\tfrac{N-1}{Nr_B}\right], \\ 
\boxed{\frac{(2^*_{N/r}-2)\,(2^*_{N/r}-1)}{2^*_{Nl/2}-1}}\quad
\text{if } 
l\ge 2\!\left[\tfrac1N+\tfrac{N-1}{Nr_B}\right].
\end{cases}
$$
\end{enumerate}
\bigskip

\item
%\label{1b:0} 
Assume $l>l_3$. Now, the
 interval is $\left(\tfrac{N}{2},\widetilde{r_B}\right)
\cap\left(\tfrac{2}{l},\widetilde{r_B}\right],$ nonempty whenever $2/l<\widetilde{r_B}\iff l>\tfrac{2}{N}+\tfrac{2(N-1)}{Nr_B},$ where we have used \eqref{rB:l}.

\begin{enumerate}[label=\textbf{(1.b.\alph*)}$_0$, leftmargin=.2cm]
\item
%\label{1ba:0} 
Now the function $\be$ is independent of $q$ on the interval \eqref{I1ba} (see \ref{1ba:q}), and 
$$
\be_0(1,b,a)=\frac{2(2-l)}{[N(1-l)+2]}(2^*_{N/r}-1),\q{for}\boxed{l>\tfrac{2}{N}+\tfrac{2(N-1)}{Nr_B}}
$$ 
see \eqref{be:1ba}.
\item
%\label{1bb:0} 
The function $\be=\be(q)$ is defined by \eqref{be:1bb}  for $q$ on the interval   $I_{1,b,b}$, given by \eqref{I1bb},  see \ref{1bb:q}.

Here, $\be$ is decreasing (see \eqref{theta-derivative}), so the infimum is reached at the minimum of the upper limits of the intervals defining $I_{1,b,b}$, either $\be(\widetilde{r_B})$,  or $\be\big(1 /\big[\tfrac{N+2}{2N}-\big(\tfrac{2}{N}-\tfrac{1}{r}\big)\tfrac{N(1-l)+2}{2(2-l)}\big]\big)$. 
Observe that, using \eqref{def:rB:tilde}
\begin{align}
 &\widetilde{r_B} \le  1 /\big[\tfrac{N+2}{2N} -\big(\tfrac{2}{N}-\tfrac{1}{r}\big)\tfrac{N(1-l)+2}{2(2-l)}\big]\nonumber\\
&\iff \tfrac{N-1}{Nr_B} \ge \tfrac{1}{2}-\big(\tfrac{2}{N}-\tfrac{1}{r}\big)\big(\tfrac{N}{2}-\tfrac{(N-2)}{2(2-l)}\big)\nonumber \\
 &\iff \tfrac{(N-2)}{2} \le \frac{\tfrac{1}{2}+\tfrac{N-1}{Nr_B}-\tfrac{N}{2r}}{\tfrac{2}{N}-\tfrac{1}{r}}\,(2-l).
 \label{1bb0:interv}
\end{align}
Observe also that a necessary condition is the following one
\begin{align}\label{NrBr}
&
\tfrac{1}{2}+\tfrac{N-1}{Nr_B}>\tfrac{N}{2r}
\iff r>\frac{N/2}{\tfrac{1}{2}+\tfrac{N-1}{Nr_B}}.
\end{align}
Under condition \eqref{NrBr}, Eq. \eqref{1bb0:interv} is equivalent to
\begin{equation}\label{l:1bb0}
l\le 2-\frac{\tfrac{2}{N}-\tfrac{1}{r}}{\frac{2}{N-2}\left[\tfrac{1}{2}+\tfrac{N-1}{Nr_B}-\tfrac{N}{2r}\right]} = \frac{\tfrac{4}{N(N-2)}+\tfrac2{N}\tfrac{2_*}{r_B}-\tfrac{(2^*-1)}{r}}{\tfrac{1}{N-2}+\tfrac1{N}\tfrac{2_*}{r_B}-\tfrac{2^*}{2r}}=:l_4.  
\end{equation}
Thus, assuming  \eqref{NrBr},
$$
\be_0(1,b,b)=\begin{cases}
\be\big(\widetilde{r_B}\big) ,&\text{if } l\le l_4 ,\\ 
\be \bigg(\frac1{\big[\tfrac{N+2}{2N} -\big(\tfrac{2}{N}-\tfrac{1}{r}\big)\tfrac{N(1-l)+2}{2(2-l)}\big]}\bigg),&\text{if } l> l_4.
\end{cases}
$$
If \eqref{NrBr} do not holds, then $\widetilde{r_B} >  1 /\big[\tfrac{N+2}{2N} -\big(\tfrac{2}{N}-\tfrac{1}{r}\big)\tfrac{N(1-l)+2}{2(2-l)}\big]$
and $\be_0(1,b,b)=\be \big(1 /\big[\tfrac{N+2}{2N} -\big(\tfrac{2}{N}-\tfrac{1}{r}\big)\tfrac{N(1-l)+2}{2(2-l)}\big]\big),$ so  for $\boxed{l>\tfrac{2}{N}+\tfrac{2(N-1)}{Nr_B}},$

$$
\be_0(1,b,b)=\begin{cases}
\boxed{(2^*_{N/r}-2)\, \frac{(2^*_{N/r}-1)}{2_{*N/r_B}-1}}\\
\qquad\quad\text{if 
} \tfrac{1}{2}+\tfrac{N-1}{Nr_B}>\tfrac{N}{2r},
\text{ and }l\le l_4,
\\[.5cm] 
\boxed{(2^*_{N/r}-2)\, \frac{(2-l)(N-2)(2^*_{N/r}-1) }{N\big(\tfrac{2}{N}-\tfrac{1}{r}\big)[N(1-l)+2]}} \\[4mm] 
\qquad\quad\text{if either } \tfrac{1}{2}+\tfrac{N-1}{Nr_B}\le\tfrac{N}{2r},\text{ or }\\[3mm]
\qquad\quad \tfrac{1}{2}+\tfrac{N-1}{Nr_B}>\tfrac{N}{2r},
\text{ and }
l> l_4.
\end{cases}
$$
\end{enumerate}
\end{enumerate}
\medskip

\item
%\label{r>rB:0}
Assume  on the other hand  $r<\widetilde{r_B} \iff \frac{1}{r}> \frac{1}{N}+\frac{N-1}{Nr_B} ,$ see \eqref{max}.

\begin{enumerate}[label=\textbf{(2.\alph*)}$_0$, leftmargin=.2cm]
\item \quad
%\label{2a:0}
Assume   $lq\le 2,$ then $\g=l.$
\begin{enumerate}[label=\textbf{(2.a.\alph*)}$_0$, leftmargin=.2cm]
\item
%\label{2aa:0} 
The function $\be=\be(q)$ is defined by \eqref{be:1aa}  for $q$ on the interval   $I_{2,a,a}$, given by \eqref{I2aa}, see \ref{2aa:q}.
Either  $l+2^*_{N/q}-2 > \tfrac{2^*}{N}-\tfrac{2_*}{r_B}
\iff q>1 /\big(\frac{l}{2^*}+\frac{1}{\widetilde{r_B}}\big),$
so using \eqref{max},  and \eqref{Def:Beta:C}, the function $\be(q)$ is defined as in \eqref{be:1aa}, for $q\in I_{2,a,a}.$
Observe that a necessary condition ensuring that the third sub-interval of $I_{2,a,a}$ is non empty, is the following one
$$
1\big/\big(\tfrac{l}{2^*}+\tfrac{1}{\widetilde{r_B}}\big)<2/l\iff l<1+(N-1)/r_B.
$$
As in case \ref{1aa:0}, if $l<1$, then $\be$ is increasing.  So the infimum is reached at the greatest lower bounds of the  intervals defining  $I_{2,a,a}$ (see \eqref{I:2aa:nonempty}), either $\be(\tfrac{N}{2})$, or $\be(2/l)$, or $\be\big(1 /\big(\tfrac{l}{2^*}+\tfrac{1}{\widetilde{r_B}}\big)\big)$. Thus, for $\boxed{l<1}$
$$
\beta_0(2,a,a)=
\begin{cases}
\boxed{l(2^*_{N/r}-1)}
& \text{if } l \le 
\min\Big\{
\tfrac{4}{N}, \tfrac{2(r_B-N+1)}{(N-2)r_B}
\Big\},
\\[1em]
\boxed{\dfrac{(2^*_{N/r}-2)\,(2^*_{N/r}-1)} {2^*_{N/r_1}-1}}
& \text{if }
\tfrac{(N+r)r_B-r(N-1)}{rr_B(N-2)} \le l < 1.
\end{cases}
$$
where 
\begin{equation}\label{Def:r0}
r_1=r_1(l,N,r_B):=\tfrac{(l(N-2)+2)r_B+2(N-1)}{2Nr_B}.
\end{equation}
If $l>1$, then  $\be$ is decreasing (see \eqref{theta-derivative}). So the infimum is reached at the minimum of the upper limits of the intervals defining  $I_{2,a,a}$, either $\be(r)$, or $\be(2/l)$, or $\be\big(1 /\big(\tfrac{l}{2^*}+\tfrac{1}{\widetilde{r_B}}\big)\big)$. 
Thus, for $ \boxed{1\le l<1+(N-1)/r_B}$

$$
\beta_0(2,a,a)=
\begin{cases}
\boxed{l+2^*_{N/r}-2}\ \text{ if }
l\le\min\!\left\{\tfrac{2}{r},
\tfrac{2[(N+r)r_B-r(N-1)]}{rr_B(N-2)}\right\},
\\[5pt]
\boxed{\tfrac{(l-1)(2^*_{N/r}-1)}{2^*_{Nl/2}-1}
+2^*_{N/r}-1}\ \text{ if }l\ge\max\!\left\{\tfrac{2}{r},
\tfrac{N-1}{r_B}+1\right\},
\\[7pt]
\boxed{\tfrac{(l-1)(2^*_{N/r}-1)}{2^*_{N/r_1}-1}+2^*_{N/r}-1} \; \;\text{if } \tfrac{2[(N+r)r_B-r(N-1)]}{rr_B(N-2)}\le l \le\tfrac{N-1}{r_B}+1.
\end{cases}
$$

\item
Either $l+2^*_{N/q}-2 \le 2^*_{N/r}
\iff q\le 1 /\left(\frac{l}{2^*}+\frac{1}{r}\right)
,$ then   using \eqref{max},  and \eqref{Def:Beta:C}, $\be(q)$  
is defined as in \eqref{be:1ab}, for $q\in I_{2,a,b}$ (see \eqref{I2ab} for a definition of $I_{2,a,b}$).

As in \ref{1ab:0}, one verifies that $\beta'(q)<0$
for all $q\in I_{2,a,b}$. Hence the infimum is attained at the least upper bound of the intervals
defining $I_{2,a,b}$ (see \eqref{I2ab}), namely
$
q=\min\big\{
r,\, \tfrac{2}{l},\,1/\big(\tfrac{l}{2^*}+\tfrac{1}{\widetilde{r}_B}\big)\big\},
$
thus
\end{enumerate}

$$
\beta_0(2,a,b)=
\begin{cases}
\boxed{2^*_{N/r}-2}\quad
\text{if }
l \le \min\!\left\{\frac{2}{r},
2^*\big[\tfrac{1}{r}-\tfrac{1}{N}-\tfrac{N-1}{Nr_B}\big]
\right\},\\[3mm] 
\boxed{\frac{(2^*_{N/r}-2)\,(2^*_{N/r}-1)}{2^*_{Nl/2}-1}}\quad
\text{if }\frac{2}{r} \le l \le 2^*\big[\tfrac{1}{r}-\tfrac{1}{N}-\tfrac{N-1}{Nr_B}\big], \\[3mm]
\boxed{\frac{(2^*_{N/r}-2)\,(2^*_{N/r}-1)}
{2^*_{N\big(\tfrac{l}{2^*}+\tfrac{1}{\widetilde{r}_B}\big)}-1}}
\quad
\text{if }
2^*\big[\tfrac{1}{r}-\tfrac{1}{N}-\tfrac{N-1}{Nr_B}\big]
\le l \le 1+\tfrac{N-1}{r_B}.
\end{cases}
$$

\item
If now we assume  also $lq> 2,$ then $\g$ is as in \eqref{ga:1b:q} and 
$(\gamma+2^*_{N/q}-2) \, \overline{\te}$
is independent of $q$ as in \eqref{ga:1b:2} for $q\in \left(\tfrac{N}{2},r\right)\cap\left(\tfrac{2}{l},r\right)$.
\end{enumerate}
\begin{enumerate}[label=\textbf{(2.b.\alph*)}$_0$, leftmargin=.2cm]
\item
Either \eqref{ineq:q:2ba} holds, then   using \eqref{max},  and \eqref{Def:Beta:C}, $\be(q)$ is independent of $q$ for $q\in I_{2,b,a}$ (see \eqref{I2ba}), so $\be_0(2,b,a)$ is also defined by  \eqref{be:1ba}.
\item
Either \eqref{ineq:q:2bb} holds,  then   using \eqref{max},  and \eqref{Def:Beta:C}, $\be(q)$
is defined by \eqref{be:2bb}, for $q\in I_{2,b,b}$ (see \eqref{I2bb}).
Let 
\begin{align}\label{r:123}
&r_2:=\tfrac{2Nr_B(2-l)}{r_B[N+2(1-l)] +(N-1)[N(1-l)+2]}, \\
& r_3:= \tfrac{4Nr_B}{(N+2)[(N-1)+r_B]}, \quad    r_4:= \tfrac{2Nr_B}{N(N-1)+2r_B}.\nonumber
\end{align}
Note that $r_2<r_3<r_4$,  that 
$r\le r_2$ if and only if 
\begin{align*}
&   rr_BN+2rr_B(1-l)+r(N-1)N(1-l)+2r(N-1)\le 2Nr_B(2-l)\\
& \iff (1-l)[2rr_B+r(N-1)N-2Nr_B]\le 2Nr_B-rr_BN-2r(N-1). 
\end{align*}
\begin{equation*}   
r\le r_2\iff \boxed{ l\le \tfrac{4Nr_B-r(N+2)\big[(N-1)+r_B\big] }{2r_B(N-r)-rN(N-1)}}, 
\end{equation*}
Moreover $\boxed{r<r_3\iff 4Nr_B>r(N+2)\big[(N-1)+r_B \big]}$, and that  $\boxed{r>r_4\iff 2Nr_B> r[N(N-1)+2r_B]}$.  
\end{enumerate}
\end{enumerate}

We  have to study, 
$
\boxed{r>r_2},
$
and the infimum is always attained at $q=r_2$, see \eqref{1ab:0} and \eqref{theta-derivative}. Let
$$
l_2:=\begin{cases}
    \tfrac{
4Nr_B-r(N+2)[(N-1)+r_B]
}{
2Nr_B-r[N(N-1)+2r_B]
} & \text{if } r<r_3\\
\tfrac{
r(N+2)[(N-1)+r_B]-4Nr_B
}{
r[N(N-1)+2r_B]-2Nr_B
} & \text{if } r>r_4.
\end{cases}
$$
then, we can get 
$$
\beta_0(2,b,b)=
\begin{cases}
    \frac{2^*}{N}-\frac{2_*}{r_B} & \text{if } r<r_3  \ \text{and}\ l>l_2\\[4mm]
    \frac{\left(\frac{2^*}{N}-\frac{2_*}{r_B}\right) (2^*_{N/r}-1)}{2^*(1-r_2)-1} & \text{if } 
\begin{aligned}
\bigl(r<r_3 \ &\text{and}\ l>l_2\bigr)
\quad\text{or}\\ 
\bigl(r>r_4 \ &\text{and}\ l<l_2\bigr)
\quad\text{or}\\
&r_3<r<r_4.
\end{aligned}
\end{cases}
$$
\end{proof}

\begin{appendix}

%----------------------------------------------------------------
%---------------------A P P E N D I X  A-------------------------
%----------------------------------------------------------------

\section{Moser iteration technique}
\renewcommand{\thesection}{\Alph{section}}
\counterwithin{equation}{section}
\setcounter{equation}{0}
\label{sec:appA}
Let us consider the following boundary value problem
\begin{equation}\label{M1.1.1}
\left \{
\begin{aligned}
-\Delta u +u=& g(x,u, \nabla u), \; \; \; x\in \Om, \\
\frac{\p u}{\p \nu }=& g_{B}(x,u), \;x\in \p \Om ,
\end{aligned}
\right.
\end{equation}
where $\Om \subset \mathbb{R}^N$, $N>2$, is an open, bounded domain with a Lipschitz boundary, and $g:\Om\times \mathbb{R}\rightarrow \mathbb{R}$ and $g_{B}:\p \Om \times \mathbb{R} \rightarrow \mathbb{R}$ are subcritical nonlinear Carathéodory functions.

The following Theorem states that any weak solution is in $L^{\infty}(\Om )\cap L^{\infty}(\p \Om )$, cf. \cite[Theorem 3.1]{Marino_Winkert_NonlAnal_2019} for a proof. 
\begin{thm}\label{th:A}
Let $u\in H^{1}(\Om )$ be a weak solution to \eqref{M1.1.1}. Let $g:\Om \times \mathbb{R} \times \mathbb{R}^{N}\rightarrow \mathbb{R}$, $g_{B}:\p \Om \times \mathbb{R}\rightarrow \mathbb{R}$ be Carathéodory functions, and
\begin{align*}%\label{M1.1.2}
|g(x,s,\xi )| &\le b_{1}|\xi|^{l} + b_{2}(x)(1+| s| ),\\
\text{and} \qquad \qquad | g_{B}(x,s)| &\le b_{B}(x)(1+|s|),
\end{align*}
where $0< l \le \tfrac{N+2}{N}$,  $b_{1}>0$, $b_{2}\in L^{\frac{N}{2}}(\Om )$ and $\; b_B\in L^{N-1}(\p \Om )$.\\
Then $u\in L^{q}(\Om )\cap L^{q}(\p \Om )$ for all $1\le q\le\infty .$ 
\end{thm}
\bigskip

%----------------------------------------------------------------
%---------------------A P P E N D I X  C-------------------------
%----------------------------------------------------------------
\section{Regularity of weak solutions}
\renewcommand{\thesection}{\Alph{section}}
%\numberwithin{equation}{section}
%\setcounter{equation}{0}
\label{sec:appB}

In this section, we establish auxiliary results on further regularity of weak solutions to \eqref{pde}, by assuming that conditions on the growth of the nonlinearities are subcritical or even critical. Using a Moser type procedure, it is known that $u\in L^{q}(\Om )\cap L^{q}(\p \Om )$ for all $1\leq q<\infty $ (see \cite[Theorem 3.1]{Marino_Winkert_NonlAnal_2019}). Moreover, using elliptic regularity theory, we state the following result that guarantees, in particular, Hölder regularity of any weak solution to \eqref{pde}. 

%********************************************************************
%************************ T H E O R E M *****************************
%********************************************************************

\begin{thm}\label{th:reg:nol}
Let $\Om \subset \mathbb{R}^{N}$, $f:\Om\times \mathbb{R} \times \mathbb{R}^N \rightarrow \mathbb{R}$ and $f_{B}:\p \Om  \times \mathbb{R}\rightarrow \mathbb{R}$ be Carath\'eodory functions, such that
\begin{align}\label{R1} 
|f(x,s, \xi)| &\le a_{1}|\xi|^{l} + \widehat{g}(x,s) \qquad \text{and} \\ |f_{B}(x,s)| &\le |a_{B}(x)| \left( 1+|s|^{2_{*,N/r_{B}}-1}\right) ,\nonumber 
\end{align}
where  $0< l \leq\tfrac{N+2}{N}$, $a_{1}\in \mathbb{R}$, $a_{B}\in L^{r_B}(\p \Om )$, $N-1<r_{B}< \infty ,$ and
\begin{align}\label{g:hat}
\widehat{g}(x,s):=|a_{2}(x)| &\left( 1+|s|^{2^{*}_{N/r}-1}\right),\ 
a_{2}\in L^{r}(\Om ),\ \frac{N}{2}<r< \infty. 
\end{align}
Let $u\in H^{1}(\Om )$ be a weak solution to \eqref{pde}, then $u\in L^{q}(\Om )\cap L^{q}(\p \Om )$ for all $1\le q\le \infty .$  \\
Moreover,   $u\in W^{1,m}(\Om )\cap  C^{\nu }(\Omb)$, for $m=m(r,r_B)$ defined by
\begin{equation}\label{def:m:r}
m=r^*\text{ if }r< \widetilde{r_B}, \q{or}   m=\frac{Nr_{B}}{N-1}\text{ if }r\ge  \widetilde{r_B},
\end{equation}
and in both cases, $m>N$.\\ 
Besides, if $0<l<\tfrac{N+2}{N}$, then the following estimates hold
\begin{equation*}%\label{estim:W1m:unificado:2}
\|u\|_{W^{1,m}(\Omega)} \le C\Big( \|\nabla u\|_{L^2(\Omega)}^{\gamma}
+ \|\widehat{g}(\cdot,u)\|_{L^{r}(\Omega)} + \|f_B(\cdot,u)\|_{L^{r_B}(\partial\Omega)}\Big),
\end{equation*}
where  $\g$ is defined as in \eqref{Def:gamma}, and
\begin{equation}\label{estim:C:nu}
\| u\|_{C^{\nu }(\Omb )}\le C\Big( \|\nabla u\|_{L^2(\Omega)}^{\gamma}
+ \|\widehat{g}(\cdot,u)\|_{L^{r}(\Omega)} + \|f_B(\cdot,u)\|_{L^{r_B}(\partial\Omega)}\Big),
\end{equation}
where  $\nu =1-\frac{N}{m}$.\\
Furthermore, 
\begin{equation*}
\|u\|_{L^{\infty}(\p\Om)}\le \|u\|_{C(\Omb)}=\|u\|_{L^{\infty}(\Om)}.   
\end{equation*}  
\end{thm}
\begin{proof}
The proof is as in \cite[Theorem B.1]{Antonio_Arciga_Pardo_Sanchez} and   
\cite[Theorem 3.1]{Marino_Winkert_NonlAnal_2019}, where the nonlinearities satisfy the structure required by Moser's iteration theorem, see Theorem \ref{th:A}. Hence,
\begin{equation}\label{u:Linf}
u\in L^{\infty }(\Omega)\cap L^{\infty }(\partial\Omega).
\end{equation} 
By hypothesis \eqref{R1}- \eqref{g:hat}, for each 
$$
u\in H_0^1(\Om)\cap L^{\infty }(\Omega)\cap L^{\infty }(\partial\Omega)
$$
$$
|f(x,u, \na u)| \le a_{1}|\na u|^{l} + |a_{2}(x)| \left( 1+|u|^{2^{*}_{N/r}-1}\right).
$$ 
Observe that  all the hypothesis of Theorem \ref{th:W1m:unif} are accomplished for 
$$
\overline{g}(x):=\widehat{g}\big(x,u(x)\big),\ \overline{g}_B(x)= |a_B(x)|(1+|u(x)|^{2_{*,N/r_B}-1}).
$$
Indeed, since \eqref{g:hat} and \eqref{u:Linf},  $\overline{g}\in L^r(\Om)$  for  $r>\tfrac{N}{2}$, and  condition \eqref{g} holds.  
Also,   $\overline{g}_B\in L^{r_B}(\p\Om)$ for $r_B> N-1$.

Consequently, $u\in W^{1,m}(\Om)$ where $m=m(r,r_B)$ is defined in \eqref{def:m:r}. Moreover, there exists some $C>0$ independent of $u$ such that\eqref{estim:W1m:unificado} holds, and by Sobolev embeddings, \eqref{estim:C:nu} holds.
\end{proof}

\section{Values of \texorpdfstring{$\beta_0$}{}}\label{ApC}
In this Section, we summarize the different definitions of the exponent 
$$\be_0=\beta_0(i,\alpha,\beta) \q{for} q\in I_{i,\al,\be},\ i=1,2,\ \al,\be\in\{a,b\},
$$
appearing in Corollary \ref{coro} according to the variation of $l,\ r$ and $r_B$. The values of $l_i$, $i=1,2,3,4$ are given by \eqref{r:<:rB}, \eqref{l:1ab}, \eqref{l3},  and \eqref{l:1bb0} respectively; the values of $r_i$, $i=1,2,3,4$ are given by \eqref{Def:r0} and \eqref{r:123}.

\begin{table}[H]
\centering
\renewcommand{\arraystretch}{1.3}
\small
\begin{tabular}{|c|c|c|}
\hline
$\bm{ I_{i,\al,\be}}$ & \textbf{Conditions} & $\bm{\beta_0(i,\al,\be)}$\\
\hline

\multicolumn{3}{|c|}{$\bm{r\ge\widetilde r_B}$}\\\hline
$(1,a,a)$&$l<\min\{ l_2,1 \}$&$\displaystyle(l-1)\tfrac{2^*_{N/r}-1}{2^*_{N(l/2^*+1/r)}-1}+2^*_{N/r}-1$\\\cline{2-3}

&$l_2\le l\le1$&$\displaystyle l(2^*_{N/r}-1)$\\\cline{2-3}

&$1<l\le l_3$&$\displaystyle (l-1)\tfrac{2^*_{N/r}-1}{2_{*N/r_B}-1}+2^*_{N/r}-1$\\\cline{2-3}

&$l>\max\{1,l_3\}$&$\displaystyle (l-1)\tfrac{2^*_{N/r}-1}{2^*_{Nl/2}-1}+2^*_{N/r}-1$\\\hline
$(1,a,b)$&$r<\tfrac{N^2/2}{\tfrac{N-1}{r_B}+1},\,l\le l_1$&$\displaystyle\tfrac{(2^*_{N/r}-2)(2^*_{N/r}-1)}{2_{*N/r_B}-1}$\\\cline{2-3}

&$l_1\le l\le N/r$&$\displaystyle\tfrac{(2^*_{N/r}-2)(2^*_{N/r}-1)}{2^*_{N(l/2^*+1/r)}-1}$\\\cline{2-3}

&$l\ge N/r$&$\displaystyle\tfrac{(2^*_{N/r}-2)(2^*_{N/r}-1)}{2^*_{Nl/2}-1}$\\\cline{2-3}

&$r\ge\tfrac{N^2/2}{\tfrac{N-1}{r_B}+1},\,l\le l_3$&same as first row\\\cline{2-3}

&$r\ge\tfrac{N^2/2}{\tfrac{N-1}{r_B}+1},\,l\ge l_3$&$\displaystyle\tfrac{(2^*_{N/r}-2)(2^*_{N/r}-1)}{2^*_{Nl/2}-1}$\\\hline

$(1,b,a)$&$l>\frac2N+\tfrac{2(N-1)}{Nr_B}$&$\displaystyle\tfrac{2(2-l)}{N(1-l)+2}(2^*_{N/r}-1)$\\ \hline
$(1,b,b)$

&
$\displaystyle
\begin{aligned}
&\tfrac12+\tfrac{N-1}{Nr_B}\le\tfrac{N}{2r}\\[1mm]
&\text{or}\quad\left(\tfrac12+\tfrac{N-1}{Nr_B}> \tfrac{N}{2r}\;\text{and}\;l\le l_4\right)
\end{aligned}
$

&
$\displaystyle
(2^*_{N/r}-2)\tfrac{(2^*_{N/r}-1)}{2_{*N/r_B}-1}
$
\\
\cline{2-3}
%.........................
&
$\displaystyle
\begin{aligned}
\tfrac12+\tfrac{N-1}{Nr_B}> \tfrac{N}{2r} \; \;\text{and}\quad l>l_4
\end{aligned}
$
&
$\displaystyle (2^*_{N/r}-2)\tfrac{ (2-l)(N-2)(2^*_{N/r}-1)}{N\left(\tfrac2N-\tfrac1r\right)[N(1-l)+2]}
$
\\
\hline

\multicolumn{3}{|c|}{$\bm{r<\widetilde r_B}$}\\ \hline
$(2,a,a)$&$l<1,\ l\le\min\{\tfrac4N,\tfrac{2(r_B-N+1)}{(N-2)r_B}\}$&$\displaystyle l(2^*_{N/r}-1)$\\\cline{2-3}

&$l<1,\ \tfrac{(N+r)r_B-r(N-1)}{rr_B(N-2)}\le l<1$&$\displaystyle\tfrac{(2^*_{N/r}-2)(2^*_{N/r}-1}{2^*_{N/r_1}-1}$\\\cline{2-3}

&$1\le l<1+\tfrac{N-1}{r_B},\,l\le\min\{\tfrac2r,\tfrac{2[(N+r)r_B-r(N-1)]}{rr_B(N-2)}\}$&$\displaystyle l+2^*_{N/r}-2$\\\cline{2-3}

&$\displaystyle\tfrac{2[(N+r)r_B-r(N-1)]}{rr_B(N-2)}
\le l \le 1+\tfrac{N-1}{r_B}$
& $\displaystyle\tfrac{(l-1)(2^*_{N/r}-1)}{2^*_{Nr_1}-1}+2^*_{N/r}-1$\\
\cline{2-3}

&$l\ge\max\{\frac2r,1+\tfrac{N-1}{r_B}\}$&$\displaystyle(l-1)\tfrac{2^*_{N/r}-1}{2^*_{Nl/2}-1}+2^*_{N/r}-1$\\\hline

$(2,a,b)$
&
$\displaystyle l\le \min\left\{ \tfrac2r,\, 2^*\left[ \tfrac1r-\tfrac1N-\tfrac{N-1}{Nr_B} \right] \right\}
$
&
$\displaystyle
2^*_{N/r}-2
$
\\
\cline{2-3}

&
$\displaystyle
\tfrac2r\le l\le 2^*\left[ \tfrac1r-\tfrac1N-\tfrac{N-1}{Nr_B} \right]
$
&
$\displaystyle
\tfrac{ (2^*_{N/r}-2)(2^*_{N/r}-1) }{ 2^*_{Nl/2}-1 }
$
\\
\cline{2-3}

&
$\displaystyle
2^*\left[ \tfrac1r-\tfrac1N-\tfrac{N-1}{Nr_B} \right] \le l\le 1+\tfrac{N-1}{r_B}
$
&
$\displaystyle
\tfrac{ (2^*_{N/r}-2)(2^*_{N/r}-1) }{ 2^*_{N\left( \tfrac{l}{2^*} +\tfrac1{\widetilde r_B} \right)}-1
}
$
\\
\hline

$(2,b,a)$
&
$\displaystyle
l>\tfrac2r
$
&
$\displaystyle
\tfrac{2(2-l)}{N(1-l)+2} (2^*_{N/r}-1)
$
\\
\hline
$(2,b,b)$&$r<r_3,\ l\le l_2$ or $r>r_4,\ l\ge l_2$&$\displaystyle\tfrac{2^*}{N}-\tfrac{2_*}{r_B}$\\\cline{2-3}
&$(r<r_3,l>l_2)$ or $(r_3<r<r_4)$ or $(r>r_4,l<l_2)$&$\displaystyle\tfrac{\left(\tfrac{2^*}{N}-\tfrac{2_*}{r_B}\right)(2^*_{N/r}-1)}{2^*(1-r_2)-1}$\\\hline
\end{tabular}
\caption{}
\label{Table1}
\end{table}

\end{appendix}

\bibliographystyle{plain}
\bibliography{ref}

\end{document}